\documentclass{siamltex}
\usepackage{amsmath,amssymb,amsfonts,graphicx,fancybox,color}
\usepackage{bm} 
\usepackage[dvipsnames]{xcolor}

\def\be#1\ee{\begin{equation}#1\end{equation}}
\newcommand{\bea}{\begin{eqnarray}}
\newcommand{\eea}{\end{eqnarray}}
\newcommand{\beas}{\begin{eqnarray*}}
\newcommand{\eeas}{\end{eqnarray*}}
\newcommand{\bfR}{\mathbb{R}}
\providecommand{\tns}[1]{\ensuremath{\mathbf{\MakeUppercase{#1}}}}
\newcommand{\bG}{\tns{L}}
\newcommand{\bhG}{\widehat{\tns{L}}}
\newcommand{\tA}{\tns{A}}
\newcommand{\tB}{\tns{B}}
\newcommand{\tP}{\tns{P}}
\newcommand{\tF}{\tns{F}}

\newcommand{\pu}{ \texttt{u}}
\newcommand{\pg}{ \texttt{g}}

\newcommand{\cA}{{\cal{A}}}
\newcommand{\mO}{\mathring{\Omega}}
\newcommand{\bfC}{\mathbb{C}}

\newcommand{\tcr}{\textcolor{black}}
\newcommand{\tcm}{\textcolor{black}}
\newcommand{\clr}{\color{black}}

\newtheorem{prop}{Proposition}
\newtheorem{remark}{Remark}
\newtheorem{example}{Example}
\newtheorem{assumption}{Assumption}
\newtheorem{alg}{Algorithm}
\newtheorem{defi}{Definition}

\title{Multi-scale S-fraction reduced-order models for massive wavefield simulations}
\author{Vladimir Druskin\footnotemark[1]%
\and Alexander V. Mamonov\footnotemark[2]%
\and Mikhail Zaslavsky\footnotemark[1]}

\begin{document}
\pagestyle{myheadings}
\maketitle

\renewcommand{\thefootnote}{\fnsymbol{footnote}}

\footnotetext[1]{Schlumberger--Doll Research Center, Cambridge, MA, USA
02139}
\footnotetext[2]{University of Houston, Houston, TX, USA 77004}

\begin{abstract}
\tcr{
We developed a novel reduced-order multi-scale method for solving large time-domain wavefield 
simulation problems.  Our algorithm consists of two main stages. During the first ``off-line'' stage the fine-grid  operator (of the graph Laplacian type} is partitioned on coarse cells (subdomains). Then projection-type multi-scale reduced order models 
(ROMs) are computed for the coarse cell operators. The  off-line stage is embarrassingly parallel 
as ROM computations for the subdomains are independent of each other. It also does not depend on the number 
of simulated sources (inputs) and it is performed just once before the entire time-domain simulation. 
At the second ``on-line'' stage the time-domain simulation is performed within the obtained multi-scale ROM framework. 
The crucial feature of our formulation is the representation of the ROMs in terms of matrix Stieltjes continued 
fractions (S-fractions). The layered structure of the S-fraction   introduces several hidden layers in the ROM  representation, that  results in the block-tridiagonal  dynamic system within each coarse cell.This allows us to sparsify the obtained multi-scale subdomain operator ROMs and to reduce 
the communications between the adjacent subdomains which is highly beneficial for a parallel implementation of 
the on-line stage. Our approach suits perfectly the high performance computing architectures, however in this 
paper we present rather promising numerical results for a serial computing implementation only. 
These results include $3D$ acoustic and multi-phase anisotropic elastic  problems.
\end{abstract}

\section{Introduction}

\tcr{
Our main goal is to speed up the simulations of large-scale time-domain wave propagation problems using
model order reduction techniques. To construct reduced order models (ROMs) we need a reference full-scale
model. Such a model comes from a spatial discretization of a second order hyperbolic PDE (or a system of PDEs)
on a reference fine grid, e.g., the scalar wave equation, linear elastodynamic system or Maxwell equations in wave regime.
}

\tcr{
Specifically, we consider the impulse source problem for a semi-discrete (discrete in space, continuous in time) 
system of linear second-order equations
}
\begin{equation}
\tA u-\tB u_{tt}=g \delta(t),\quad u|_{t<0}=0.
\label{eqn:wave_fll}
\end{equation}
where $0\ge\tA=\tns{A}^T\in\bfR^{N\times N}$, $u(t),g\in\bfR^{N}$, $\tB$ is a diagonal matrix with positive 
elements and $\delta(t)$ is the Dirac delta function. {\clr For example,  (\ref{eqn:wave_fll}) can be obtained from  conservative discretization of   acoustic wave equation
\[\nabla \cdot \left[\sigma \nabla \pu(x,t)\right] -\rho \pu_{tt}(x,t)=\pg(x)\delta(t), \quad \pu|_{t<0}=0,\]
on fine grid, with $x\in \bfR^3$ and variable coefficients $\sigma(x)$ and $\rho(x)$ are respectively medium stiffness and density, in  which case $\tns{A}$ is the graph-Laplacian operator and elements of $\tB$ are masses of the grid cells.}

 We assume that the reference fine grid operator $\tA$ is 
sufficiently sparse with nontrivial diagonal entries. Its selfadjoint property implies reflective  boundary conditions, 
e.g., Dirichlet or Neumann ones, though our approach allows  more general absorbing conditions including the 
perfectly-matched layer (PML).

We target remote sensing applications, e.g., seismic exploration, where the wavefield is excited by {\it locally} 
supported sources (inputs) $g\in\bfR^n$ and measured by {\it locally} supported receiver weight distributions 
$q\in\bfR^n $(outputs). For simplicity we assume that they are approximated by discrete $\delta$-functions 
(columns of the identity matrix). In addition,  solution of (\ref{eqn:wave_fll}) is required as a convolution with 
a waveform of limited frequency bandwidth $ [0,\omega_{max}]$ assuming  $\omega_{max}\ll \sqrt {\|A\|}$.
This is consistent with the condition usually required for a good approximation of the continuous problem 
\cite{Marburg2008}.

To fix the idea, we assume that the reference fine grid operator $\tA$ is obtained via second-order finite-differences 
or finite element discretization. In general, such methods  do not require to have a grid conforming to the medium 
discontinuities and allow (sometimes with some loss of convergence order) to use different homogenization and 
error correction techniques to handle curvilinear interfaces, e.g., \cite{Moskow,zaslperg,Leveque}. However, they 
require many of degrees of freedom per wavelength for accurate enough approximation.  In various important 
applications, such as seismic exploration, this results in a sparse operator $\tA$ of dimension of order 
$10^6-10^9$ or even higher.

High-order and spectral methods allow to reduce the size $N$ significantly, however, grid (or elements) adaptation 
conforming to the interfaces is required to maintain high accuracy. Not only this makes gridding technique complicated, 
but may also worsen Courant-Friedrich-Levy (CFL) condition for explicit time-stepping stability. We also note that 
parallelization of low-order methods is significantly more efficient, thanks to the sparsity of the resulting $\tA$.

Conventional multi-scale (MS) approaches are targeted to applications where the discretization of spatial operators 
has to be performed over large-scale domain with multiple small-scale heterogeneities and where accurate 
approximation of fine-scale effects is not required. Variations of MS methods include MS finite elements, or 
superelements \cite{ efendiev, fedorenko, houwu,babuskaosborn}, MS finite volume \cite{tchelepi} as well as 
averaging algorithms \cite{Moskow, zaslperg}. MS methods allow to capture small-scale effects of the composite 
media on medium-scale computational grid (so-called coarse grid). 
Implementation of MS methods is typically split into an off-line part of computing the approximation space at 
each coarse grid cell and an on-line part of solving the problem discretized on a coarse grid. The first stage is 
embarrassingly parallel, i.e. each coarse grid cell can be handled independently. The on-line part requires 
communications between coarse grid cells, however the cost is rather marginal, thanks to low order of 
approximation at each coarse cell. 

MS methods have been successfully applied for solution of elliptic and parabolic PDEs, particularly with application 
to reservoir simulations. However, for hyperbolic problems the application of low-order MS methods is limited 
unless the coarse-scale problem is  oversampled in terms of degrees of freedom per wavelength. Likewise, 
conventional averaging techniques can be applied only when   coarse  cells are much smaller
compared to wavelength. Therefore, high-order MS methods are required for accurate discretization of spatial 
operators in wave problems. First {\it spectrally accurate} multi-scale finite element method for wavefield simulations 
has been developed in \cite{efendiev_acoust}. Spectral convergence is achieved by spanning the approximation 
space for each coarse cell on  a specially chosen {\clr subspace} of solutions of the Dirichlet problem 
in that cell. Similar to the traditional MS 
methods, the approximation problems at each coarse grid cell can be solved independently at the off-line stage. 
{\clr However, the communication and arithmetical  costs of the on-line part is affected by appearance  of {\it fully dense blocks} of the size of those subspaces, responsible  for approximation of the wave dynamics   within each coarse cell and its interaction with the neighbors.}

In this work we  address  the above problem, by extending so-called optimal (spectrally matched) grids, 
also known in the literature as finite-difference Gaussian qudrature rules \cite{druskin2000gaussian,asvadurov2000application}. 
This approach yields discretizations with a sparsity pattern of second order finite-difference 
schemes but with spectral convergence at the  corners of the rectangular coarse cells. However, the applications of 
the optimal grids are limited to the media with uniform   coarse cells (below we also call them subdomains). 
The method presented here (first outlined in  \cite{dmz_sfmsfv, dmz_mmrom}) avoids this limitation and allows 
for arbitrary sharp discontinuities within the subdomains. We will call it the multi-scale S-fraction reduced-order 
model  (MSSFROM), because the sparsification is based on Stieltjes continued fraction (a.k.a.  S-fraction) 
representation of the reduced order models (ROMs) of the Neumann-to-Dirichlet maps (NtD) of the subdomains. 
{\clr The layered structure of the S-fraction   introduces several hidden layers in the ROM  representation, that  results in the block-tridiagonal  dynamic system within each coarse cell, thus the communication between coarse cells only involves the state variables from the external layers. This is the  crucial new feature of the MSSFROM compared to the original MS formulation of \cite{efendiev_acoust,Efendiev_elas}.}
To forestall further discussion, we can refer to Figure~\ref{fig:romstenfll} for graphical 
introduction of the MSSFROM concept. In fact, this concept can be traced back to an approach to synthesis of 
complex networks known in electrical engineering and control communities at least from 1960's, when passive 
electrical networks with desired properties are assembled via a system of passive sub-networks (filters) obtaining 
via rational approximation of some a priori given impedance’s, e.g., see \cite{matthaei1964microwave}.


\tcr{
Our starting point is a large stiffness  matrix $\tA$ {\clr (of graph-Laplacian type)} and a diagonal mass matrix $\tB$ obtained from the 
reference PDE discretization on a {\it fine} grid. We begin the first, off-line stage by partitioning the fine 
computational grid into multiple subdomains and constructing the corresponding partitions of operators 
$\tA$, $\tB$ using a divide-and-conquer graph partitioning algorithm. Hereafter we also refer to the subdomains 
as the coarse grid cells, by analogy with conventional MS methods. For each coarse grid cell we separate the 
interior and boundary nodes and reduce the entire full-scale problem to the problem on the 
boundaries of coarse grid cells. The elimination of interior degrees of freedom and further conjugation of 
adjacent coarse cells is performed in terms of discrete Neumann-to-Dirichlet maps (NtD), defined via Schur 
complement-like matrices.
}

The NtD in the frequency or time domain can be viewed as a multi-input/multi-output  (MIMO) transfer function 
\tcr{of a dynamical system}. Model order reduction theory provides well-developed tools to construct 
low-dimensional 
approximations of such functions. In this work we adopt the ideas of frequency-limited model reduction that 
targets an a priori chosen frequency-range \cite{gawronski}. It allows to incorporate the information on limited 
signal bandwidth into the construction of optimal rational approximant of the transfer function. The obtained 
reduced-order model (ROM) has significantly fewer degrees of freedom, however, its matrix representation is 
fully dense. This is similar to replacing a low-order discretization with a spectral one. {\clr We show then that  the 
obtained coarse cell ROM  can be expressed in block tridiagonal form thanks to layered structure of the  S-fraction representation. 
The advantage of such transformation is twofold. First,  the conjugation of the adjacent cells  results in the 
 network approximation with matrix coefficients of the size of the S-fraction layers. Second, at on-line stage,  exclusion of  hidden layers of the S-fraction from the conjugation condition} allows to 
minimize not only the cost of algebraic operations with reduced operator, but also the cost of communication 
between adjacent cells. The latter is highly important for implementation on high-performance computing (HPC) 
architectures where the communication cost often dominates the cost of computation.

The paper is organized as follows.
In section \ref{graphpart} we partition the fine-grid dynamical system into the subproblems on coarse cells. 
In section \ref{sect:elimvar} we eliminate the interior unknowns at each coarse cell, and the problem is 
reformulated in terms of Neumann-to-Dirichlet maps (boundary transfer functions) of the coarse cells. 
Then, in section \ref{sect:rom} we discuss in details the model reduction technique used for the approximation 
of Neumann-to-Dirichlet maps of coarse cells. In particular, we compress the inputs and outputs 
(see section \ref{sect:compio}) and then construct a reduced-order model of the transfer function 
(see section \ref{sect:rom_compio}). In section \ref{sparserom}, we obtain sparse (block-tridiagonal) realization 
of the the coarse cell ROM by equivalently transforming it to a matrix S-fraction. The global multi-scale ROM is 
obtained by coupling these sparse ROM realizations via conjugation conditions in section \ref{sect:conj}. 
Then we summarize our approach in section \ref{sect:methsum} and in section~\ref{Stieltjesness} we prove 
Stieltjes property (stieltjesness) of the obtained multi-scale ROM, that implies its stability and energy conservation 
of its minimal realization. A number of numerical examples showing efficiency of our approach are presented 
in section \ref{sect:numex}. In the appendix we describe in more detail the coarse cell partitioning algorithm {\clr for graph-Laplacian type of operators}
(Appendix~\ref{sect:divconq}) and the efficient computation of the S-fraction coefficients via the block-Lanczos 
algorithm (Appendix~\ref{appC}).

\section{Grid partitioning and operator splitting} 
\label{graphpart}

\tcr{
The very first step of the off-line stage is to partition the reference fine grid into the subdomains (coarse cells)
and obtain the corresponding splitting of the fine grid stiffness matrix $\tA$ and mass matrix $\tB$.
For this we will need the following elementary definitions from graph theory.
}

Let $a_{kl}$,  $1 \le k,l \le N$ be the elements of $\tA$ and $\Omega$ be the nodal set of the graph associated 
with $\tA$, i.e., the set of reference fine grid nodes. The nodes $k,l \in \Omega$ are called adjacent iff $a_{kl}\ne 0$.   
For any $\mO \subset \Omega$ we call its neighborhood $\cA(\mO)$ the set of all adjacent nodes of $\mO$.
If $\tA$ is a full matrix, then $\forall \mO \subset \Omega$ $\cA(\mO)=\Omega$. In the other extreme case of 
diagonal $\tA$, obviously, $\cA(\mO)=\mO$. 

We assume, that  $\Omega$ can be partitioned into $N_c$ ($N_c\ge 2$) nonempty (possibly intersecting)
subsets $\Omega_i$ with $N_i$ nodal points (fine grid nodes) in each that we refer to as the subdomains
or coarse cells. For all $i \ne j $ the subdomains $\Omega_i$ and $\Omega_j$ are coupled via their 
intersection $\Gamma_{ij} = \Omega_i \cap \Omega_j$, i.e.,
\be
\label{part} 
\Gamma_{ij} = \cA(\mO_i)\cap \cA(\mO_j),
\ee 
where
$\mO_i = \Omega_i \setminus \Gamma_i$ and 
\tcr{$\Gamma_i = \bigcup_{j \ne i} \Gamma_{ij}$}.
We will call $\mO_i$ and ${\clr \Gamma_{i}}$ \tcr{respectively the interior and the boundaries of 
subdomain $\Omega_i$}.

We will limit our consideration to the case when the outputs and inputs and are supported on the partitioning  
``skeleton'' $\Gamma=\bigcup_{j=1}^{N_c}\Gamma_j$, i.e.,  
\be
\label{gsupport} 
q|_{\mO}=0,\ g|_{\mO}=0, \qquad \mO=\bigcup_{j=1}^{N_c} \mO_j.
\ee
\tcr{
\begin{remark} In addition to its main objective to minimize the arithmetical cost and communications, 
our domain partitioning algorithm is constrained by condition (\ref{gsupport}) with inputs $g$ and outputs $q$  
(approximated by discrete $\delta$-functions) given a priori.
\end{remark}
}

We define prolongation operators ${\clr \tP_i \in \bfR^{N \times N_i}}$ via the action on $x \in \bfR^{N^i}$ as
\tcr{\be
\label{eqn:prolong}
\left( \tns{P}_i x \right)_k = \left\{ \begin{tabular}{ll}
$x_k$, & if  $k \in \Omega_i$ \\
$0$, &  otherwise
\end{tabular}\right.
\ee }
We split the operators \tcr{$\tA$ and $\tB$} respectively on $\tA_i, \tB_i\in \bfR^{N_i\times N_i}$ defined 
on subdomains $\Omega_i$ so that for any $\omega \in \bfC$ we have
\be 
\label{splitN_c} 
\tns{A}+\omega^2\tB=\sum_{i=1}^{N_c}{\tns{P}_i(\tns{A}_i+\omega^2\tB_i)\tns{P}_i^T}.
\ee
	
We require $\tA_i$ and $\tB_i$ to be nonnegative and positive definite matrices respectively. 
Let us consider some simple examples, illustrating such splitting.

\begin{example} \label{1d_example}
Consider a 1D wave equation 
\begin{equation}
\label{eqn:lapl1d}
u_{xx}-u_{tt}=g
\end{equation} 
on an interval $(-1,1)$ with homogeneous Dirichlet boundary conditions on $\partial\Omega$. Let the equation be 
discretized on a uniform finite-difference grid $\{x_i = ih\}^{n}_{i=-n}$ using a second-order scheme in space 
$\tns{A}u-u_{tt}=g$, where 
$$
\left(\tns{A}u\right)_k = \frac{1}{h}\left(\frac{u_{k+1}-u_k}{h}-\frac{u_{k}-u_{k-1}}{h}\right), 
\quad u_{-n}=u_n=0.
$$ 
Let $\Omega_1 = \{ k \;|\; -n \le k \le 0\}$, $\Omega_2 = \{ k \;|\; 0 \le k \le n \}$ and, consequently, 
$\Gamma=\{k=0\}$. 
The matrix $\tns{A}\in\bfR^{(2n-1)\times (2n-1)}$ in this case is diagonally dominant.

\tcr{The split operators are defined as follows:}
\begin{align*}
\left(\tns{A}_1 u\right)_k = & \left(\tns{A}u\right)_k, \quad k<0,\\
\left(\tns{A}_1 u\right)_0 = &-\frac{1}{h}\left(\frac{u_{0}-u_{-1}}{h}\right),\\
\left(\tns{A}_2 u\right)_k = & \left(\tns{A}u\right)_k, \quad k>0,\\
\left(\tns{A}_ 2u\right)_0 = & \frac{1}{h}\left(\frac{u_{1}-u_0}{h}\right).
\end{align*}

Moreover, since $a_{00}=-\sum_{k \ne 0} a_{0k}$,  this is the unique splitting with nonnegative 
definite $\tA_i$ for the given partitioning.
 
The splitting of $\tns{B}$ (which is the identity matrix here) into $\tns{B}_1\in\bfR^{n\times n}$ and 
$\tns{B}_2 \in \bfR^{n \times n}$ is not unique. Indeed, it can be chosen as 
\begin{align*}
\left(\tns{B}_1 u\right)_k = & 1, \quad k<0, \quad \left(\tns{B}_1 u\right)_0=\beta,\\
\left(\tns{B}_2 u\right)_k = & 1, \quad k>0, \quad \left(\tns{B}_2u \right)_0=1-\beta,
\end{align*}
with any  $\beta \in (0,1)$. From the discretization point of view, this non-uniqueness is caused by an ambiguity in 
choosing the dual grids nodes, i.e. the nodes to which the approximation of the derivative is assigned.
This is equivalent to discretizing (\ref{eqn:lapl1d}) in $\Omega_1$ and $\Omega_2$ independently with Neumann 
conditions at $\Gamma$, with the first dual steps scaled by $\beta$ and $1-\beta$ respectively. 
We choose $\beta=\frac{1}{2}$ to approximately balance the condition numbers of the splitted 
matrix pencils.
\end{example}


\begin{example}
\label{2dom_ex}
Consider a 2D wave equation 
\be
\label{eqn:lapl}
\Delta u - u_{tt} = g 
\ee
in a rectangular domain $(-1;1)\times(0;1)$ with homogeneous Dirichlet boundary conditions on 
$\partial\Omega$. Let the equation be discretized on a uniform finite-difference grid 
\tcr{$$
\left\{(x_p, y_q) \;|\; x_k=ph, \; y_l=qh \right\}^{n,n}_{k=-n,l=0}
$$}
using a second-order scheme in space $\tns{A}u-u_{tt}=g$, where the action of the symmetric operator 
$\tns{A} \in \bfR^{(2n-1)(n-1) \times (2n-1)(n-1)}$ is defined as
\tcr{\be
\begin{split}
\left(\tns{A}u\right)_{p,q} = & 
        \frac{1}{h}\left(\frac{u_{p+1,q}-u_{p,q}}{h}-\frac{u_{p,q}-u_{p-1,q}}{h}\right) \\
+ & \frac{1}{h}\left(\frac{u_{p,q+1}-u_{p,q}}{h}-\frac{u_{p,q}-u_{p,q-1}}{h}\right)
\end{split}
\ee}
with boundary conditions $u_{p,0}=u_{p,n}=u_{-n,q}=u_{n,q}=0.$

We begin with the two-domain partitioning of 
\tcr{$$
\Omega = \{ (p,q) \;|\; p=-n,\ldots,n; \; q=0,\ldots,n\}
$$}
and define the splitting as follows. Let 
\tcr{$$
\Omega_1 = \{ (p,q) \;|\; -n \le p \le 0 \}, \quad \Omega_2 = \{ (p,q) \;|\; 0 \le p \le n \}
$$}
and, consequently, \tcr{$\Gamma = \{(0, q) \;|\; q = 0, \ldots, n \}$}. 
\tcr{The split operators are defined as}
\begin{align*}
\left(\tns{A}_1 u\right)_{p,q} = & \left(\tns{A} u\right)_{p,q}, \quad p<0,\\
\left(\tns{A}_1 u\right)_{0,q} = & -\frac{1}{h}\left(\frac{u_{0,q}-u_{-1,q}}{h}\right)
+ \frac{0.5}{h}\left(\frac{u_{0,q+1}-u_{0,q}}{h}-\frac{u_{0,q}-u_{0,q-1}}{h}\right)
\end{align*}
with boundary conditions $u_{p,0}=u_{p,n}=u_{-n,q}=0$ and
\begin{align*}
\left(\tns{A}_2 u\right)_{p,q} = & \left(\tns{A} u\right)_{p,q}, \quad k>0,\\
\left(\tns{A}_2 u\right)_{0,q} = & \frac{1}{h}\left(\frac{u_{1,q}-u_{0,q}}{h}\right)
+ \frac{0.5}{h}\left(\frac{u_{0,q+1}-u_{0,q}}{h}-\frac{u_{0,q}-u_{0,q-1}}{h}\right)
\end{align*}
with boundary conditions $u_{p,0}=u_{p,n}=u_{n,q}=0$;
\begin{align*}
\left(\tns{B}_1 u\right)_{p,q} = & 1, \quad p<0, \quad \left(\tns{B}_1 u\right)_{0,q}=\frac{1}{2}, \\
\left(\tns{B}_2 u\right)_{p,q} = & 1, \quad p>0, \quad \left(\tns{B}_2 u\right)_{0,q}=\frac{1}{2}.
\end{align*}
In this scenario $\tA$ is diagonally dominant, and so is its splitting \tcr{$\tA_1$, $\tA_2$}.
\end{example}

\begin{example}
\label{multidom_ex}
Now let us consider a multi-domain splitting of 
$\tns{A}, \tns{B}$ from Example~\ref{2dom_ex} on a regular square coarse grid. 
To distinguish the coordinate directions explicitly, we use here the notation with two indices. 
We define for $i=1,\ldots,N^x_c$ and $j=1,\ldots,N^y_c$ the partitioning 
$$
\Omega_{i,j} = \{(p,q) \;|\; p = n_x (i-1)-n+1, \ldots, n_x i-n, \; q = n_y(j-1)+1,\ldots, n_y j \}
$$
of $\Omega=\{(p,q) \;|\; p=-n,\ldots,n, \; q=1,\ldots,n \}$. 
Here $n_x N^x_c = 2n$, $n_y N^y_c = n$ and $N^x_c N^y_c = N_c$ where $N^x_c, N^y_c$ 
are the numbers of coarse cells (subdomains) in each direction and $n_x$, $n_y$ are the dimensions 
of coarse cells. Partitioning skeleton is given by
\clr{
\begin{eqnarray} 
\nonumber \Gamma=\{(n_x i-n,j) \;|\; i=1, \ldots, N^x_c-1, \; j=1, \ldots ,n \}\cup\\
 \{(i, n_y j) \;|\; i=-n, \ldots, n , \; j=1, \ldots, N^y_c-1 \}.
\end{eqnarray}}

We split the operators $\tA$ and $\tB$ as 
\begin{align*}
\left(\tns{A}_{i,j} u\right)_{p,q} = & \frac{\gamma_x}{h}
\left( \alpha^+_x \frac{u_{p+1,q}-u_{p,q}}{h} - \alpha^-_x \frac{u_{p,q}-u_{p-1,q}}{h}\right) + \\
& \frac{\gamma_y}{h} \left( \alpha^+_y \frac{u_{p,q+1}-u_{p,q}}{h} - \alpha^-_y \frac{u_{p,q}-u_{p,q-1}}{h}\right), \\
\left(\tns{B}_{i,j} u\right)_{p,q} = & \beta_x \beta_y u_{p,q},
\end{align*}
for $(p,q) \in \Omega_{i,j}$, where the coefficients 
$\alpha^+_x, \alpha^-_x, \alpha^+_y, \alpha^-_y, \gamma_x, \gamma_y, \beta_x, \beta_y$ 
are equal to $1$ for $(k,l) \notin \Gamma$ and are defined as follows otherwise 
\begin{itemize}
\item if $p = n_x (i-1) - n$ then $\alpha^+_x = 1$, $\alpha^-_x=0$, $\gamma_y=0.5$, $\beta_x = 0.5$
\item if $p = n_x i - n$ then $ \alpha^+_x = 0$, $ \alpha^-_x = 1$, $\gamma_y=0.5$, $\beta_x = 0.5$
\item if $n_x (i-1) - n < p <n_x i - n$ then $\alpha^+_x = \alpha^-_x =0.5$, $\gamma_y=1$, $\beta_x = 1$
\item if $q = n_y (j-1) $ then $\alpha^+_y = 1$, $\alpha^-_y = 0$, $\gamma_x=0.5$, $\beta_y = 0.5$
\item if $q = n_y j$ then $\alpha^+_y = 0$, $\alpha^-_y = 1$, $\gamma_x=0.5$, $\beta_y = 0.5$
\item if $n_y (j-1) < q < n_y j$ then $\alpha^+_y = \alpha^-_y = 0.5$, $\gamma_x=1$, $\beta_y=1$
\end{itemize}
\end{example}

To optimize the cost and simplify the partitioning, we usually consider partitioning into regular coarse grids with 
cubic cells $\Omega_i$  (corresponding to actual subdomains of the computational domain), similar to the 
one shown in Example~\ref{multidom_ex}. Generally, such partitioning and splitting is an intuitively obvious 
procedure for $\tns{A}$ obtained from a discretization of an elliptic operator in regular domains with regular 
fine grids. 

However, for more complicated discretizations and vectorial PDEs (e.g., elasticity), one can perform graph 
partitioning via Algorithm~\ref{alg2} from Appendix~\ref{sect:divconq}. By construction it yields diagonal 
$\tB_i\succ 0$ and for (non-strictly) diagonally dominant $\tA$ it guarantees (non-strict) diagonal dominance 
of $\tA_i$ (see Proposition~\ref{prop2} in Appendix~\ref{sect:divconq}).

\section{Elimination of interior nodes via boundary transfer functions}
\label{sect:elimvar}


\tcr{
In the previous section we defined the interior $\mO_i$ and boundary nodes $\Gamma_i$ for each 
subdomain $i=1,\ldots,N_c$. The next step is to eliminate the interior nodes $\mO_i$. This
is done via a boundary transfer function, a Neumann-to-Dirichlet map defined at the boundaries $\Gamma_i$.
A natural setting to define boundary transfer functions is the frequency (Laplace) domain.
}

We consider the Laplace-domain counterpart of (\ref{eqn:wave_fll_fd}) (with some abuse of notation 
we denote the time and frequency-domain solutions by $u$)
\begin{equation}
\tns{A}u+\omega^2 \tB u = g,
\label{eqn:wave_fll_fd}
\end{equation}
with Laplace frequency $\omega$ satisfying the condition 
\be
\label{noninspectr}
-\omega^2\in \bfC \setminus \left(\bfR_- \cup \{0\} \right). 
\ee 
This condition assures non-singularity of $\tns{A} + \omega^2 \tB$ and  $\tns{A}_i + \omega^2 \tB_i$, 
$i=1,\ldots,N_c$.
From (\ref{splitN_c}) we obtain 
\be
\label{wave_split_fd}
\sum_{i=1}^{N_c}{\tns{P}_i(\tns{A}_i+\omega^2\tB_i)\tns{P}_i^T}u=g.
\ee

\tcr{
We define the prolongation operator $\tP^{\Gamma}_i\in\bfR^{N_i\times K_i}$ from 
$\Gamma_i$ to $\Omega_i$ as
\be
\label{eqn:prolongomega}
\left( \tns{P}^{\Gamma}_i x \right)_k = \left\{ \begin{tabular}{ll}
$x_k$, & if  $k \in {\clr \Gamma_i}$ \\
$0$, & otherwise
\end{tabular} \right., \quad \mbox{for } x \in \bfR^{K_i}.
\ee 
}
Here we assume that $N_i \gg K_i$, where $K_i$ is the the number of nodes in $\Gamma_i$. This assumption is 
targeted to the case when $\tA$ is obtained via a low order discretization of a second order elliptic PDE operator 
or system. In particular, the case of second order discretizations on regular grids yields $\Gamma_i$ 
\tcr{that are one reference fine grid node ``thick''}, see Example~\ref{multidom_ex}. 
  
Using (\ref{gsupport}) we eliminate the interior nodes from (\ref{wave_split_fd}), by transforming it to 
\be
\label{wave_split_fd_el} 
\sum_{i=1}^{N_c}{\tns{P}_i \tP^{\Gamma}_i 
\left ( \tns{F}_i (\omega^2) \right)^{-1} 
\left( \tns{P}_i \tP^{\Gamma}_i \right)^T} u = g,
\ee

where 
\be
\label{tfcur}
\tns{F}_i(\omega^2)={\tns{P}^\Gamma_i}^T(\tns{A}_i+\omega^2\tns{B}_i)^{-1}\tns{P}^\Gamma_i
\in \bfC^{K_i\times K_i},
\ee
is the so-called matrix transfer function, a.k.a. Neumann-to-Dirichlet map, Weyl or impedance matrix valued function 
which is closely related to Schur complements of $\Omega_i$. The stieltjesness of $\tns{F}_i$ (more precisely of its 
equivalent transform $\tns{F}_i(\omega^2)$\ ) will play a fundamental role in our further derivations. We will use the 
following definition of the Stieltjes matrix-valued functions introduced by Mark Krein for the scalar case, e.g. see \cite{bolotnikov1999operator,fritzsche2015matrix, arlinskii2011conservative}.

\begin{defi}
\label{stieltjes}
Let $\tF: \bfC \setminus (0, +\infty) \rightarrow \bfC^{p\times p}$.
Then $\tF$ is called a
$(0, +\infty)$-Stieltjes function of order $p$ if $\tF$ satisfies the following three conditions:
\begin{enumerate}
\item $\tF$ is holomorphic in $\bfC \setminus (0, +\infty)$.
\item For all $z$ with $\Im z>0$ the matrix $\Im(\tF)$ is non-negative definite Hermitian.
\item For all $z\in (−-\infty, 0)$ the matrix $\tF(z)$ is non-negative definite Hermitian.
\end{enumerate}
\end{defi}

The transfer function can be written via the spectral decomposition of the matrix pencil 
$(\tns{A}_i, \tns{B}_i)$ as  
$$ 
\tns{F}_i(\omega^2)=\sum_{l=1}^{N_i} \frac{V_l^i{V^i_l}^T}{\lambda_l^i+\omega^2},
$$
were $\lambda_l^i\in\bfR_-$ are the eigenvalues of matrix pencil $(\tns{A}_i,\tns{B}_i)$ and  $V_l^i\in\bfR^ {K_i}$ 
are the restrictions of the corresponding eigenvectors on $\Gamma_i$. It is easy to see that $\tns{F}_i(\omega^2)$ 
satisfies all three conditions of Definition~\ref{stieltjes}, so it is {\it the Stieltjes matrix valued function} of 
$\omega^2$. Moreover, it can be easily shown via the Schur complement formula, that 
$\left( \tns{F}_i(\omega^2) \right)^{-1}$ is linear with respect to 
$(\mathring{\tns{A}}_i+\omega^2\mathring{\tns{B}}_i)^{-1}$, where 
$\mathring{\tns{A}}_i\le 0$ and $\mathring{\tns{B}}_i>0$  are the diagonal blocks respectively 
of $\tns{A}_i$ and $\tns{B}_i$ (equivalently, of $\tns{A}$ and $\tns{B}$) corresponding to 
${\mO}_i$. This allows to replace non-negative definiteness with positive-definiteness in 
conditions~2 and 3 of Definition~\ref{stieltjes}, that yields the equivalence of (\ref{wave_split_fd}) and 
(\ref{wave_split_fd_el}).

Denote 
$$
u|_{\Gamma_i} = {\tP_i^\Gamma}^T{\tP_i}^T u \in \bfR^{K_i},
$$ 
where ${\tP_i^\Gamma}^T{\tP_i}^T$ is obviously the restriction operator from $\Omega$ to $\Gamma_i$. 
Then we can equivalently rewrite (\ref{wave_split_fd_el}) in terms of  the boundary restrictions of $u$ as
\be 
\label{BIE}
\sum_{i=1}^{N_c} \tP_i \tP_i^\Gamma \left( \tF_i(\omega^2) \right)^{-1} u|_{\Gamma_i} = g.
\ee

For $N_c>2$ there may exist a nonempty ``corner'' set $\Gamma^{\cap}=\bigcup_i\bigcup_{j \ne j'} \Gamma_{ij}\cap\Gamma_{ij'}.$ 

In the case of regular partitioning, similar to the one described in the Example~\ref{multidom_ex}, dimension of such 
set is $o(\sum_{i=1}^{Nc} K_i)$. {\clr We also note that from a topological point of view the ``corner'' set in $3D$ regular cells includes not only corners but also the entire edges. To simplify the derivation, we consider the case with empty ``corner'' set}
\be
\label{nocorners} 
\Gamma^{\cap}=\emptyset.
\ee  
The fine grid discretization can be modified to satisfy such condition by introducing hanging nodes and removing 
nodes from $\Gamma^{\cap}$ {\clr (see Appendix~\ref{sect:cor_set_rem}). We show how it works explicitly for example \ref{multidom_ex}. 

\begin{example}
\label{cs_multidom_ex}
Let us consider the removal of the corner set $$
\Gamma^{\cap}=\{(n_x i-n,n_y j) \;|\; i=1, \ldots, N^x_c-1, \; j=1, \ldots, N^y_c-1 \}$$ of the operator from Example \ref{multidom_ex}. For each $(p;q)\in\Gamma^{\cap}$, we introduce four hanging nodes $(p,q^u), (p,q^l), (p^r,q), (p^l,q)$ and replace $(\tA u)_{p,q}$ and $(\tB u)_{p,q}$ by 
\begin{align*}
(\tA' u)_{p,q^u}= & \frac{1}{2h}
\left(\frac{u_{p+1,q^u}-u_{p,q^u}}{h} - \frac{u_{p,q^u}-u_{p-1,q^u}}{h}\right) +
\frac{1}{h} \left(\frac{u_{p,q^u+1}-u_{p,q^u}}{h}\right), \\
\left(\tns{B}' u\right)_{p,q^u} = &0.5u_{p,q^u},\\
(\tA'u)_{p,q^l}= & \frac{1}{2h}
\left(\frac{u_{p+1,q^l}-u_{p,q^l}}{h} - \frac{u_{p,q^l}-u_{p-1,q^l}}{h}\right) +
\frac{1}{h} \left(\frac{u_{p,q^l-1}-u_{p,q^l}}{h}\right), \\
\left(\tns{B}' u\right)_{p,q^l} = &0.5u_{p,q^l},\\
(\tA' u)_{p^r,q}= &\frac{1}{h} 
\left(\frac{u_{p^r+1,q}-u_{p^r,q}}{h}\right)+
  \frac{1}{2h}\left(\frac{u_{p^r,q+1}-u_{p^r,q}}{h} - \frac{u_{p^r,q}-u_{p^r,q-1}}{h}\right)\\
\left(\tns{B}' u\right)_{p^r,q} = &0.5u_{p^r,q},\\
(\tA' u)_{p^l,q}= &\frac{1}{h} 
\left(\frac{u_{p^l-1,q}-u_{p^l,q}}{h}\right)+
  \frac{1}{2h}\left(\frac{u_{p^l,q+1}-u_{p^l,q}}{h} - \frac{u_{p^l,q}-u_{p^l,q-1}}{h}\right)\\
\left(\tns{B}'u\right)_{p^l,q} = &0.5u_{p^l,q}.
\end{align*}

$\tB'$ is diagonal and $\tA'$ is assumed to be symmetric: we define its action on vectors with non-zero elements at hanging node $(p;q)$ by symmetry via $(\tA' u)_{p,q}$. In Figure~\ref{fig:nocornerset} we show the modified graph of $\tA$ corresponding to the partitioning from 
Example~\ref{multidom_ex} with the corner set removed.
\end{example}

The same approach can be straightforwardly extended to  the case of $3D$ (and multidimensional)   Cartesian grids, where ``corner'' set also include the edges of  the coarse cells. 
For  regular enough solutions with uniformly bounded first differences  such a modification, obviously  leads to $O\left(\frac{b_{kk}}{\mathrm{Tr}\hspace{1pt} B}\right)$ error, where $\frac{b_{kk}}{\mathrm{Tr}\hspace{1pt} B}$ is the relative mass of the corner point.  When both the fine and coarse grids are uniform Cartesian,  the total  relative mass  of the corner points is inversely proportional to the coarse cell diameter squared. Thus,  for large enough coarse cells removal of the corner points adds the error, consistent with the second order discretization error on the fine grid.
In all our numerical experiments this modification 
results in a negligible error, however solution correction can be made at each time step, if needed 
(see Remark \ref{rem:solcorr}).

\begin{figure}[htb]
\centering
\includegraphics[width=\columnwidth]{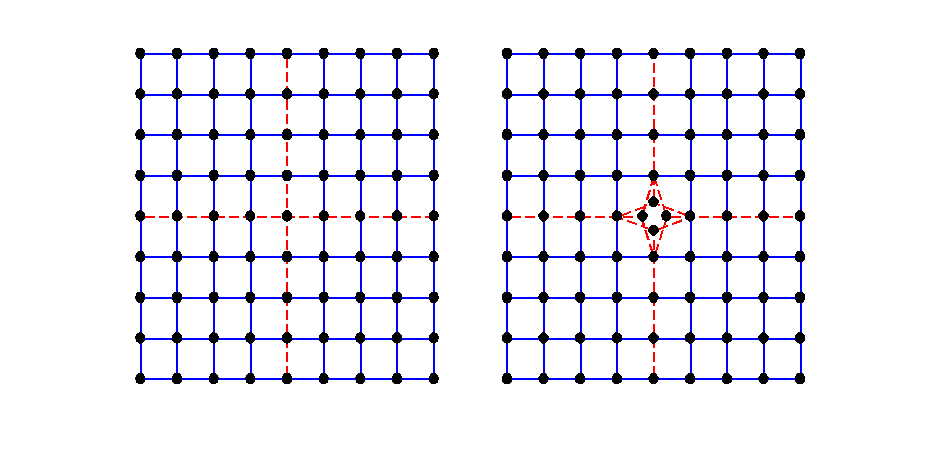}
\caption{Graph of $\tA$ from Example \ref{multidom_ex} (left) and its modification without corner set (right). 
$\Gamma$ is shown by red lines. $\Gamma^\cap$ consists of one node (center node on the left plot).
}
\label{fig:nocornerset}
\end{figure}

Introducing a small error and considering the case (\ref{nocorners}) we can simplify the system (\ref{BIE}) to }
\be
\label{BEINC}
\tP_{ij}^T \left( \tF_i (\omega^2) \right)^{-1} u|_{\Gamma_i} + 
\tP_{ji}^T \left( \tF_j (\omega^2) \right)^{-1} u|_{\Gamma_j} = 
g|_{\Gamma_{ij}}, \quad \forall (i,j):\;  i\ne j,\; \Gamma_{ij}\ne\emptyset, 
\ee
where $\tP_{ij}^T\in\bfR^{K_{ij}\times K_i}$  is the restriction operator from $\Gamma_i$ to 
$\Gamma_{ij}$ and $K_{ij}$ is the dimension of $\Gamma_{ij}$.

\section{\tcr{Reduced order models on subdomains}}
\label{sect:rom}


After the partitioning of $\Omega$ and the corresponding splitting of $\tns{A}, \tns{B}$ is computed and the 
matrix transfer functions (\ref{tfcur}) are defined on the subdomain boundaries $\Gamma_i$, the next point to 
address is the efficient computation of transfer functions $\tF_i(\omega^2)$, $i=1,\ldots, N_c$. To this effect we 
replace the true $\tF_i(\omega^2)$ with reduced-order models.
The complexity of the matrix-valued transfer function can be described by its McMillan degree, i.e., the dimension of the equivalent minimal realization via the first order dynamical system \cite{mcmillan}.
Thus, the McMillan degree of the exact transfer function $\tF_i(\omega^2)$ is generally equal  to $2N_i$.

Model order reduction is a two step process.

First, in section~\ref{sect:compio} we consider the reduction of dimensionality of boundary degrees of freedom 
from $K_i$ to $\widetilde{K}_i \ll K_i$. This is achieved by projecting $\tF_i(\omega^2)$ on a tall skinny orthogonal 
matrix $\tns{S}_i\in \bfR^{K_i\times {\widetilde K}_i}$, $\tns{S}_i^T \tns{S}_i = \tns{I}$, which is obtained via an 
approximate proper orthogonal decomposition of the boundary restrictions of solutions $u|_{\Gamma_i}$. 
The resulting projected transfer function has the form
\be
\label {project}
\widetilde{\tF}_i(\omega^2)=\tns{S}_i^T \tF_i(\omega^2) \tns{S}_i \in\bfR^{{\widetilde K}_i\times {\widetilde K}_i}.
\ee

Second, in section~\ref{sect:rom_compio} we compute matrix Pad\'e approximant 
$\widetilde{\tF}^m_i(\omega^2)$ of $\widetilde{\tF}_i(\omega^2)$ 
via Pad\'e-Krylov connection satisfying $2m$ matching conditions.

As we shall see in section~\ref{sect:conj}, the first step decreases  both the communication and  arithmetical cost, 
while the second step mainly reduces the arithmetical cost. 
As a result of these two steps, we reduce the McMillan degree of the final transfer function  from $2N_i$ to $2m\widetilde K_i$.

\subsection{Proper orthogonal decomposition of boundary solution restrictions}
\label{sect:compio}

The column space of $\tns{S}_i$ should produce a good approximation basis for the boundary restrictions of 
all possible propagative solutions. Optimal compression, in principle, can be obtained via the 
proper orthogonal decomposition (POD) on $\Gamma=\bigcup _{i=1}^{N_c}\Gamma_i$ of frequency bounded 
solutions, or equivalently, via frequency-limited balanced truncation \cite{gugercin2004survey}. Here we 
implemented an approximate approach of \cite{DKSZ} that is outlined below.

Following \cite{asvadurov2000application}, we will concentrate on the propagative part of the solution for 
the cutoff frequency $\omega_{max}$, that for $t \ge 0$ is given by spectral decomposition 
\be
\label{specdec} 
q^T u = \sum_{\lambda_l \le \omega_{max}^2} 
\frac{\sin{\sqrt{\lambda_l}t}}{\sqrt{\lambda_l}} q^T z_l z_l^T g,
\ee
where $\lambda_l \in \bfR$ and $z_l \in \bfR^N$, $l=1,2,\ldots,N$ are respectively the eigenvalues and the 
eigenvectors of matrix pencil $(\tA,\tB)$. Accurate approximation of the propagative part also gives reasonably 
good approximation of the entire solution. e.g., see \cite{asvadurov2000application}.

	
{\clr The subspace of propagative solution restricted on $\Gamma_{ij}$ can be approximately computed via the frequency-limited Gramian:
\be
\label{trgr}
G^{\omega_{max}}_{ij}=\sum_{\lambda_l\le \omega_{max}^2} v_l v_l^T \in \bfR^{K_{ij}\times K_{ij}}, 
\quad v_l = z_l|_{\Gamma_{ij}},
\ee
here $z_l$ are the eigenvectors of matrix pencil (\tns{A}, \tns{B}).	
We denote by $\tns{S}_{ij}\in \bfR^{M_{ij}\times K_{ij}}$ the matrix of eigenvectors of $G^{\omega_{max}}_{ij}$
corresponding to the eigenvalues above some truncation threshold $\epsilon>0$. 
Let the set $J(i)$ be defined as
\tcr{
$$
J(i) = \{ j \;|\; j \ne i \mbox{ and } \Gamma_{ij}\ne\emptyset \}.
$$
}
We construct $\tns{S}_i\in \bfR^{K_i \times \widetilde K_i}$,
\tcr{$\widetilde K_i=\sum_{j\in J(i)}  M_{ij}$}
as a block diagonal matrix with $\tns{S}_{ij}$, $j \in J(i)$ as the diagonal (non-square) blocks, i.e., 
$$
\tns{S}_i \equiv
    \begin{bmatrix}
        \tns{S}_{ij_1} & 0     &    0    &  \ldots  \\
	0 & \tns{S}_{ij_2}      & 0 &  \ldots  \\\ldots
	  & 0& \tns{S}_{ij_3}  & 0\\
	  &  \ldots      & 0      & \tns{S}_{ij_{jmax}}
    \end{bmatrix},$$
where $J(i) = \{ j_1,\ldots, j_{jmax}\}$.

Assuming fast enough decay of the eigenvalues of $G^{\omega_{max}}_{ij}$, it can be shown that 
$$
\| \tns{S}_i \widetilde{\tF}_i(\omega^2) \tns{S}_i^T -{\tF}_i(\omega^2) \|_{\infty} < O(\epsilon).
$$
It is known that for large enough $N_i$ the Gramian truncation procedure exhibits spectral convergence.

Direct computation  via spectral decomposition (\ref{trgr}) is prohibitively expensive for sufficiently large   problems. However, $G^{\omega_{max}}_{ij}$ can be equivalently rewritten as composition of spectral projection and prolongation operators
\be\label {trgrstep}
 \left( \tns{P}_i \tP^{\Gamma}_i\tns{P}_{ij} \right)^T\eta  \left(\omega_{max}^2\tns I-\tns {B}^{-1}\tns{A}\right)\tns{P}_i \tP^{\Gamma}_i \tns{P}_{ij},
\ee
where $\eta(\lambda) $ is the Heaviside step-function, i.e.,  $\eta(\lambda)=1$ if $\lambda\ge0$ and $0$ otherwise.   We compute spectral projector $\eta  \left(\omega_{max}^2\tns I-\tns {B}^{-1}\tns{A}\right)$  via optimal (Zolotarev)  rational approximations of the sign function, e.g., see   \cite{güttel2015zolotarev,di2016efficient,nakatsukasa2015computing}.
The Zolotarev approximant usually requires up to $20$ linear shifted solves, that can still be expensive with  large  $N$.
To speed up the Zolotarev algorithm we use 
an approximate Gramian ${\widetilde G}^{\omega_{max}'}_{ij}$ obtained for the matrix
 $\widetilde\tB _{ij}^{-1}\widetilde \tA_{ij}$ of a 
smaller dimension 
but with the same block corresponding to $\Gamma_{ij}$. The low dimensional operator  $\widetilde\tB _{ij}^{-1}\widetilde \tA_{ij}$ is chosen to get accurate approximation of  the Schur complement of full problem at
$\Gamma_{ij}$.
Complete description  of this approach will appear  in  \cite{DKSZ}. }
\subsection{\tcr{Reduced order model for the projected transfer function}}
\label{sect:rom_compio}

Having the compressed set of inputs and outputs $\tns{S}_i$ in hand, the efficient approximation of 
the projected transfer function 
\tcr{
$$
\widetilde{\tns{F}}_i(\omega^2) = 
\tns{S}_i^T {\tP_i^\Gamma}^T (\tns{A}_i+\omega^2\tns{B}_i)^{-1} \tP_i^\Gamma \tns{S}_i,
$$}
can be obtained using traditional projection-based model reduction methods. 
\tcr{These methods seek the ROM of $\widetilde{\tns{F}}_i(\omega^2)$ by replacing $\tns{A}_i$ and $\tns{B}_i$
with projections on some subspace ${\cal{V}}_m$.} Note that projection subspace ${\cal{V}}_m$ may be different 
for different coarse cells. However, for brevity, we omit the index $i$ of ${\cal{V}}_m$. The rational Krylov subspace (RKS) 
is one of the most popular choices of projection subspaces in model order reduction community, allowing to obtain 
ROMs as multipoint Pad\'e approximants with exponential convergence rates \cite{ruhe}.   

In our numerical examples we stick to block rational Krylov subspace\footnote{See remark~\ref{Sinsubspace} for 
discussion whether to include zero power term $\tns{S}_i$ in the subspace.}
\tcr{
\be
\label{npks}
{\cal{V}}_m = \mbox{colspan}
\{\tP_i^\Gamma \tns{S}_i, (\tns{A}_i+s \tns{B}_i)^{-1} \tP_i^\Gamma \tns{S}_i, \ldots, 
(\tns{A}_i+s\tns{B}_i)^{-m+1} \tP_i^\Gamma \tns{S}_i \} 
\ee}
with some small (in absolute value) negative shift $s$. This RKS can be implemented by applying a block 
Lanczos iteration to the pair \tcr{$\left( (\tns{A}_i+s\tns{B}_i)^{-1}, \tP_i^\Gamma \tns{S}_i \right)$}.

\begin{assumption}\label{ass1} 
\tcr{Here we assume that all $(\tns{A}_i+s\tns{B}_i)^{-k} \tP_i^\Gamma \tns{S}_i$ are linearly independent 
for $k=0,\ldots,m-1$}, i.e. $\dim({\cal{V}}_m) = m \widetilde K_i$. This assumption generally holds for problems 
arising from discretization of uniformly bounded elliptic PDEs provided $m \widetilde K_i \ll N_i$, but 
theoretically may break in the case of elliptic systems with degeneration, e.g., in the mix of solid and liquid for 
the elasticity system. This would require block Lanczos method with deflation, but we did not observe such 
phenomena in practice. 
\end{assumption}

Note that the computation of  ${\cal{V}}_m$ requires multiple linear solves with shifted matrix $\tns{A}_i$ that 
can be quite costly. However, the cost is alleviated by several factors. First, the computation is only done on 
small subdomains. Second, the computations for different subdomains are independent of each other, 
thus they can be performed in parallel \tcr{without any data exchange}. Third, the computation only must be 
done just once before the time stepping for all sources, \tcr{what we refer to as the off-line computation}. 
Also, a precomputed Cholesky factorization can be reused for the repeated linear solves. 

\begin{remark} Projection on subspace (\ref{npks}) is easy to implement and it provides
$\widetilde{\tns{F}}_i^m$ as Pad\'e approximants of $\widetilde{\tns{F}}_i$ as function of $\omega^2$ at 
frequency $s$ with exponential convergence for 
\tcr{$\omega^2\in \bfC\setminus\lambda[\tB_i^{-1} \tA_i]$}. 
However, it may not be optimal in terms of the number of  degrees of freedom  per wavelength. Other model
reduction techniques such as time- and/or frequency-limited balanced truncation \cite{gugercin2004survey} 
may be more appropriate for these purposes, however their implementation can be more expensive.
\end{remark} 

Let \tcr{$\tns{V}_m \in \bfR^{N_i \times m \widetilde{K}_i}$} 
be an orthonormal basis for ${\cal{V}}_m$. 
\tcr{Reduced-order model of the projected transfer function can be obtained as}
\be
\label{tfprj}
\widetilde{\tns{F}}^m_i(\omega^2) = 
(\tns{S}^m_i)^T(\tns{A}^m_i+\omega^2\tns{B}^m_i)^{-1}\tns{S}^m_i 
\in \bfR^{\widetilde{K}_i\times \widetilde{K}_i}
\ee
\tcr{where $\tns{A}^m_i = \tns{V}_m^T \tns{A}_i \tns{V}_m \in \bfR^{m \widetilde{K}_i \times m \widetilde{K}_i}$ 
and $\tns{B}^m_i = \tns{V}_m^T \tns{B}_i \tns{V}_m \in \bfR^{m \widetilde{K}_i \times m \widetilde{K}_i}$ 
is the projected matrix pencil and 
$\tns{S}^m_i = \tns{V}_m^T \tP_i^\Gamma \tns{S}_i \in \bfR^{m \widetilde{K}_i \times \widetilde{K}_i}$} 
is the projected input-output.

\begin{lemma}
\label{lemStil}
Matrix valued function $\widetilde{\tns{F}}^m_i(s)$ is the $(0,\infty)$-Stieltjes function of order $\widetilde{K}_i$. 
\end{lemma}
\begin{proof}
Projected matrices $\tns{A}^m_i$ and $\tns{B}^m_i$ are  symmetric and respectively non-positive and 
positive definite, so their pencil has eigenpairs
$\theta_l^i \in \bfR_-$,  $y_l^i\in \bfR^{m \widetilde{K}_i}$, \tcr{$l =1,\ldots, m \widetilde{K}_i$}.
From the spectral decomposition we obtain a partial fraction representation
\be
\label{partfr} 
\widetilde{\tns{F}}^m_i(\omega^2) = 
\sum_{l=1}^{m\widetilde K_i}\frac{v_lv_l^T}{\theta_l^i+\omega^2}, 
\quad v_l={P_i^{\Gamma}}^T y_l^i\in \bfR^{\widetilde K_i}.
\ee
Obviously, representation (\ref{partfr}) satisfies all 3 conditions of Definition~\ref{stieltjes}.
\end{proof}

\begin{assumption}\label{S-fractionrep} 
It is known that the McMillan degree of the reduced order model $\widetilde{\tns{F}}^m_i(s)$ 
(that can also be equivalently defined as twice the minimal number of terms in representation (\ref{partfr}), 
\cite{Callier1991}) is bounded by twice the dimension of the projection space 
(in the second order formulation, i.e. by $2m \widetilde{K}_i$). We assume that the actual 
McMillan degree of $\widetilde{\tns{F}}^m_i(s)$ is equal to this bound, which is what we observed in our 
experiments given in section~\ref{sect:numex} for sufficiently large $N_i$.
\end{assumption}
 
\section{\tcr{Equivalent sparse representations of ROMs}}
\label{sparserom}

\tcr{
The ROM (\ref{tfprj}) in general involves the projected matrix pencil 
$\tns{A}^m_i, \tns{B}^m_i \in \bfR^{m \widetilde{K}_i \times m \widetilde{K}_i}$ with dense matrices. This leads 
to high computational costs at the on-line stage when the time stepping is performed with such operators. To reduce
the computational cost of the on-line stage we need to compute the equivalent minimal realization of the ROM via a 
linear time-invariant dynamic system. In this section we construct a sparse realization of such system, which mimics 
the structure of a three-point second order finite-difference scheme. }

The approach of this section is a direct generalization of the so-called finite-difference quadrature rules a.k.a. 
spectrally matched or optimal grids \cite{druskin2000gaussian,asvadurov2000application} from scalar to matrix 
transfer functions. For brevity, in this section we omit index $i$ assuming that we currently consider a single 
coarse cell $\Omega_i$ only.


The core of our approach is an equivalent representation of $\widetilde{\tns{F}}^m(\omega^2)$ via matrix Stieltjes 
continuous fraction (S-fraction) form \cite{bolotnikov1999operator} as
\be
\label{S-fraction}  
\widetilde{\tns{F}}^m(\omega^2) = 
\frac{1}{ -\frac{\omega^2}{\widehat{\bG}^1} + \dfrac{1}{\frac{1}{{\bG}^1} +
\dfrac{1}{\ddots \; +\dfrac{1}{-\frac{\omega^2}{\widehat{\bG}^m} + \bG^m}}}},
\ee
where $\widehat \bG^k, \bG^k \in \bfR^{\widetilde K\times\widetilde K}$ are symmetric positive-definite matrices
\tcr{and the division is understood as matrix inversion.} 
As follows from results of \cite{dyukarev2004indeterminacy}, Assumption~\ref{S-fractionrep} 
guarantees the existence and uniqueness of the S-fraction representation (\ref{S-fraction}) of Stieltjes 
$\widetilde{\tns{F}}^m(\omega^2)$ obtained (as in the previous section) via the matrix Pad\'e approximant.

A direct recursive algorithm transforming $\widetilde{\tns{F}}^m$ from the partial fraction form (\ref{partfr})
to the S-fraction form (\ref{S-fraction}) is Algorithm 7.2 in \cite{druskin2015direct}. It also can be viewed as a 
constructive existence proof, that if such a representation exists, it is unique and automatically produces positive
definite $\widehat{\bG}^k$, $\bG^k$, $k = 1,2,\ldots,m$ for Stieltjes $\widetilde{\tns{F}}^m$. Another, more 
complicated but also more computationally efficient algorithm via the block-Lanczos method is given in 
Appendix~\ref{appC}.
 
  The novelty of our approach is in following original Krein's representation of the scalar S-fraction via a string of point 
masses and weightless springs, a so-called Stieltjes string \cite{kac1974spectral}, to represent 
$\widetilde{\tns{F}}^m(\omega^2)$ equivalently via a transfer function of a 
``three-point finite-difference scheme'' \tcr{with matrix coefficients}
\bea
\bG^1 (\tns{U}^2 - \tns{U}^1)  + 
\omega^2 (\bhG^1)^{-1} \tns{U}^1 & = \tns{I}, 
\label{eqn:GammaU1st}\\
\bG^k (\tns{U}^{k+1} - \tns{U}^k) -
\bG^{k-1} (\tns{U}^{k} - \tns{U}^{k-1})  +
\omega^2 (\bhG^k)^{-1} \tns{U}^k & = \tns{0},
\label{eqn:GammaUrem}
\eea
by introducing fictitious (full rank) matrix variables $\tns{U}^k \in \bfC^{\widetilde{K} \times \widetilde{K}}$, 
$k = 1,2,\ldots, m+1$, with $\tns{U}^{m+1} = 0$, so that
\be 
\label{eqn:tfgamma}
\widetilde{\tns{F}}^m(\omega^2) = \tns{U}^1.
\ee
\begin{remark}
\label{sfint}
In the case $\widetilde K=1$, (\ref{S-fraction}) becomes a scalar S-fraction and 
(\ref{eqn:GammaU1st})--(\ref{eqn:GammaUrem}) 
becomes the regular second-order finite-difference scheme on the optimal staggered grid with primary and 
dual steps $1 / \bG^j$ and $1 / \bhG^j$  respectively.

For multidimensional problems ($\widetilde K>1$), \tcr{the columns of $\tns{U}^1$} 
give the projections of the 
solutions on the subdomain boundary, and \tcr{the columns of $\tns{U}^k$ for $k>1$ have the same dimension 
$\widetilde{K}$}, 
i.e., they can be viewed as projections of solutions on some virtual concentric layers inside $\Omega_i$ and  
(\ref{eqn:GammaU1st}--\ref{eqn:GammaUrem}) can be interpreted as a three-point second order discretization 
in the subdomain interior in the direction normal to these layers. A schematic representation of the of the stencil 
inside $\Omega_i$ is shown in Figure~\ref{fig:romsten}.
The S-fraction representation is a known tool in the homogenization theory (e.g., see \cite{markel2012homogenization}). 
Indeed, if subdomain has dimensions much less compared to the wavelength then conventional non-dispersive 
isotropic or anisotropic effective medium can be constructed. Otherwise, the effective medium has to be 
frequency-dependent (dispersive). In this contest (\ref{S-fraction}) can be viewed as an effective dynamic impedance 
of the subdomain converted to an equivalent network approximation via (\ref{eqn:GammaU1st})-(\ref{eqn:tfgamma}). 
\end{remark}

\begin{figure}[htb]
\centering
\includegraphics[width=0.45\columnwidth]{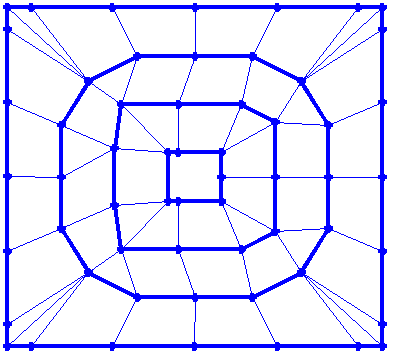}
\caption{Schematic representation of the stencil for the s-fraction ROM inside single subdomain. Each layer of unknowns 
is connected with its neighbors only. Here we only outline  just some  possible nodes and connections: 
Connection between neighbors is dense in general, i.e.,  every unknown is one layer is connected with all unknowns inside 
that layer as well as all unknowns in the neighboring layers;  The number  of the unknowns at every layer is the same as at 
the coarse cell boundary, i.e.,  $\widetilde K$, see remark~\ref{sfint}.}
\label{fig:romsten}
\end{figure}

\section{Boundary conjugation}
\label{sect:conj}

\tcr{
Construction of the ROM $ \widetilde{\tF}_i^m$ and its S-fraction representation (\ref{S-fraction}) concludes
the off-line stage. At the on-line stage the solutions at neighboring subdomains need to be conjugated via the
reduced order transfer functions on the shared boundaries. This leads to data exchange between the subdomains. 
In this section we derive the formulas for solution conjugation which dictate the data to be exchanged.}

Replacing  the true full-scale transfer function $\tns{F}_i(\omega^2)$ in (\ref{BEINC}) with ROM (\ref{eqn:tfgamma}) 
we obtain
\tcr{\be 
{\widetilde\tP_{ij} }^T \left( \widetilde{\tF}_i^m(\omega^2) \right)^{-1} U^1_i +
{\widetilde\tP_{ji} }^T \left( \widetilde{\tF}_j^m(\omega^2) \right)^{-1} U^1_j = 
\widetilde{g}_{ij} 
\label{BEINCr}, 
\ee}
where 
$$
\widetilde{g}_{ij}=\tns{S}_{ij}^T g|_{\Gamma_{ij}},
$$
and $\widetilde{\tns{P}}^T_{ij} \in \bfR^{\widetilde K_{ij} \times \widetilde{K}_i}$ is the restriction matrix from 
the set of reduced inputs-outputs $\tns{S}_{i}$ to the set that corresponds to $\tns{S}_{ij}$.

Vector variable $U^1_i\in\bfR^{\widetilde K_i}$ (a function of $\omega$) has the meaning of the projected approximate 
solution of (\ref{BEINC}), i.e. 
$$
u |_{\Gamma_{i}}\approx \tns{S}_{i} U^1_i.
$$
Hence, the solution at receiver $q$ can be approximated by 
\be
\label{solatrec} 
q^T u \approx \sum_{i,j} \widetilde{q}_{ij} \widetilde{u}_{ij},
\ee 
where 
$$
\widetilde u_{ij} = \widetilde{\tns{P}}_{ij}^T U_i^1,
$$ 
and 
$$
\widetilde q_{ij} = \tns{S}_{ij}^T q|_{\Gamma_{ij}}.
$$
We impose conjugation conditions 
\be
\label{contred}
\widetilde u_{ji}=\widetilde u_{ij}
\ee 
for every two adjacent subdomains $\Omega_i$ and $\Omega_j$. 
By construction $\tns{S}_{ij} = \tns{S}_{ji}$, so (\ref{contred}) leads to conjugation for the approximate 
solutions at $\Gamma_{ij}$. From (\ref{eqn:GammaU1st}) and (\ref{eqn:tfgamma}) 
\tcr{for subdomain $\Omega_i$ we obtain
$$
\left( \bG^1_i (\tns{U}^2_i - \tns{U}^1_i) \right) + \omega^2 (\bhG^1_i)^{-1} \tns{U}^1_i  = 
\left( \widetilde{\tns{F}}^m_i (\omega^2) \right)^{-1} \tns{U}^1_i.
$$
The same holds for $\Omega_j$}.
As it was mentioned above, matrices $\tns{U}^k_i$ are of full rank, so 
for $\Omega_i$ we can write
\tcr{$$ 
\left( \bG^1_i ({U}_i^2 - {U}_i^1) \right) + \omega^2 \left({\bhG^1}_i\right)^{-1}{U}_i^1  = 
\left(\widetilde{\tns{F}}_i^m(\omega^2)\right)^{-1}{U}^1_i,
$$}
where ${U}_i^2\in \bfR^{\widetilde K_i}$. The same relation holds for the $\Omega_j$.
Then equation (\ref{BEINCr}) can be rewritten as 
\be
\begin{split}
\widetilde{\tns{P}}_{ij}^T \left( \bG^1_i (U^2_i - U^1_i) \right) + 
\widetilde{\tns{P}}_{ji}^T \left( \bG^1_j (U^2_j - U^1_j) \right)  & + \\
\omega^2 \left( \widetilde{\tns{P}}_{ij}^T (\bhG_i^1)^{-1} U^1_{i} + 
                          \widetilde{\tns{P}}_{ji}^T (\bhG_j^1)^{-1} U^1_{j} \right) & = 
\widetilde{g}_{ij}. 
\end{split}
\label{eqn:firsttstep_conj} 
\ee
  
The system  will be completed by adding to conjugation conditions (\ref{contred})--(\ref{eqn:firsttstep_conj})
and internal equations from (\ref{eqn:GammaUrem}) for all $\Omega_i$:
\be
{\bG}_i^k ({U_i}^{k+1} - {U_i}^k) -
{\bG}_i^{k-1} ({U_i}^{k} - {U_i}^{k-1})  +
\omega^2 (\bhG_i^k)^{-1} {U_i}^k  = 0, 
\quad U_i^{m+1}=0,
\label{eqn:GammaUremi}
\ee
\tcr{where $k=2,3,\ldots,m$}.

Replacing $\omega^2$ by $-\frac{d^2}{dt^2}$ we transform (\ref{eqn:firsttstep_conj})--(\ref{eqn:GammaUremi}) 
to the semi-discrete time-domain system (using the same notation $U_i^k$ for the time-domain solution)
\bea
\bG^k_i (U^{k+1}_i - U^k_i) -
\bG^{k-1}_i (U^{k}_i - U^{k-1}_i) - \frac{d^2}{dt^2}(\bhG^k_i)^{-1} U^k_i & = &0, \nonumber\\ 
\quad U_i^{m+1}=0, \quad k=2,3,\ldots,&m,&
\label{eqn:intrtstep_td}
\eea
\be
\begin{split}
\widetilde{\tns{P}}_{ij}^T \left( \bG^1_i (U^2_i - U^1_i) \right) + 
\widetilde{\tns{P}}_{ji}^T \left( \bG^1_j (U^2_j - U^1_j) \right) & - \\ 
\frac{d^2}{dt^2} \left( \widetilde{\tns{P}}_{ij}^T (\bhG_i^1)^{-1} U^1_{i} + 
                                       \widetilde{\tns{P}}_{ji}^T (\bhG_j^1)^{-1} U^1_{j} \right) & =
\delta(t) \widetilde g_{ij},  
\end{split}
\label{eqn:firsttstep_conj_td}
\ee
coupled with conjugation condition (\ref{contred}) and initial conditions
\be
\label{initial}
U_i^1|_{t<0}=0, 
\ee
for $i=1,\ldots,N_c$ and for all $j \ne i$ such that $\Gamma_{ij} \ne \emptyset$.

\begin{example}
In 1D Example \ref{1d_example}, the discretization 
(\ref{eqn:intrtstep_td}), (\ref{eqn:firsttstep_conj_td}), (\ref{contred}) 
corresponds to the conventional finite-difference scheme on the spectrally-matched grid
\tcr{$$
\frac{1}{\widehat{h}^k} \left( \frac{U^{k+1} - U^k}{h^k} - \frac{U^{k} - U^{k-1}}{h^{k-1}}\right) + 
\frac{d^2}{dt^2} U^k = g_0 \delta_{k 0}, \quad k = -n+1, \ldots, n-1,
$$}
with $U^{-n} = U^{n} = 0$. Here $\delta$ is Kronecker's delta, 
\tcr{split ROM solution $U$ is indexed as 
$U_1^{-k} = U^{k-1}$, $k = -n, \ldots, -1$ and $U_2^{k} = U^{k-1}$, $k = 1,\ldots,n$, 
and the grid steps $\{h^k\}^{n-1}_{k = -n}$, $\{\widehat{h}^k \}^{n-1}_{k = -n+1}$}
are defined as 
\tcr{\begin{align*}
h^k = & \left( \bG_1^{-k} \right)^{-1}, \quad k = -n, \ldots, -1, \\
h^k = & \left( \bG_2^{k+1} \right)^{-1}, \quad k = 0, \ldots, n-1, \\
\widehat{h}^k = & \left( \bhG_1^{-k+1} \right)^{-1}, \quad k = -n+1, \ldots, -1, \\
\widehat{h}^k = & \left( \bhG_2^{k+1} \right)^{-1}, \quad k = 1, \ldots, n-1, \\
\widehat{h}^0 = & \left( \bhG_1^{1} + \bhG_2^{1} \right)^{-1}.
\end{align*}}
In Figure~\ref{fig:1d_grd} we show an example of a 1D spectrally-matched grid. In particular, we used
ROMs of size 7 and 11 for the coarse grid cells in $\Omega_1 = \{k \;|\; -n \le k \le 0 \}$ and 
$\Omega_2 = \{ k \;|\; 0 \le k \le n \}$, respectively. \tcr{We refer to $\{h^k\}^{n-1}_{k=-n}$ as the primary grid 
nodes and $\{\widehat{h}^k\}^{n-1}_{k=-n+1}$ as the dual grid nodes. Solution $U^k$ is defined at the 
primary grid, while the derivative $(U^{k+1}-U^k)/h^k$ is defined at the dual grid nodes.}
\end{example}

\begin{figure}[htb]
\centering
\includegraphics[width=\columnwidth]{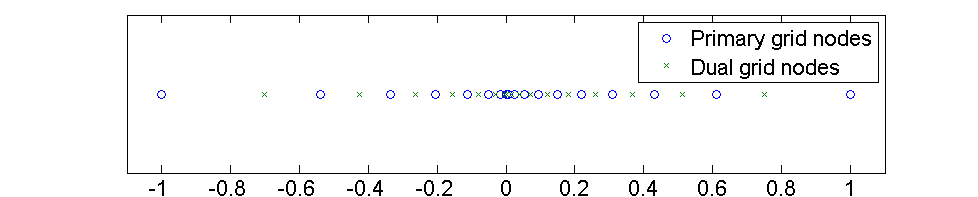}
\caption{A 1D spectrally-matched grid. \tcr{Blue circles are the primary (solution) grid nodes, 
green crosses are the dual (derivative) grid nodes.}}
\label{fig:1d_grd}
\end{figure}

\begin{remark} 
\label{Sinsubspace} 
Note that in general equations (\ref{eqn:firsttstep_conj_td}) define a
time-stepping scheme for the boundary solutions $U^1$ with a global mass matrix. The exception here 
is when matrices $\bhG^1_i$ are block diagonal such that their non-zero elements are 
${\widetilde{\tns{P}}_{ij}}^T \bhG^1_i \widetilde{\tns{P}}_{ij}$, $j \in {\cal{J}}(i)$ only.
This can be achieved by adding $\tns{P}_i^\Gamma \tns{S}_i$ 
to the projection subspace. Then for the solution at the shared boundary 
$\widetilde u_{ij} = \widetilde u_{ji} = \widetilde{\tns{P}}_{ij}U^1_{i} = \widetilde{\tns{P}}_{ji}U^1_{j}$ 
equation (\ref{eqn:firsttstep_conj_td}) is replaced by
\begin{equation}
\begin{split}
\frac{d^2 \widetilde u_{ij}}{dt^2} = &
\left( 
\left({\widetilde{\tns{P}}_{ij}}^T \bhG^1_i \widetilde{\tns{P}}_{ij} \right)^{-1} + 
\left({\widetilde{\tns{P}}_{ji}}^T \bhG^1_j \widetilde{\tns{P}}_{ji} \right)^{-1} \right)^{-1} \times \\
& \times \left(
\widetilde{\tns{P}}_{ji} \bG^1_j (U^2_j - U^1_j)+ 
\widetilde{\tns{P}}_{ij} \bG^1_i (U^2_i - U^1_i)
\right) + \delta(t) \widetilde{g}_{ij},
\end{split}
\label{eqn:bdrytstep}
\end{equation}
which only requires communication between the adjacent subdomains.
{\clr Mass-matrix diagonalization by adding new elements to the projection subspace cannot increase the approximation error, in fact, it  slightly improves accuracy.    However, it may tighten the CFL stability limitation for the explicit 
time-stepping, which anyway will be still looser compared to the original problem (\ref{eqn:wave_fll}), see remark~\ref{MFEM}.}
Any standard time stepping scheme can be used for (\ref{eqn:intrtstep_td}), (\ref{eqn:bdrytstep})  
including Runge-Kutta, etc. The expressions on the right hand side of (\ref{eqn:intrtstep_td}), (\ref{eqn:bdrytstep}) 
are always evaluated at the current time step.
\end{remark}

\begin{remark} 
Every equation (\ref{contred}), (\ref{eqn:intrtstep_td}), (\ref{eqn:bdrytstep}) contains at most three vector unknowns 
$U_i^j$ and $\widetilde u_{ij}$, that can be interpreted as boundary restrictions of the solution, see Remark~\ref{sfint}. 
Schematically the stencil of the obtained operator is shown in Fig. \ref{fig:romstenfll}. So, in general the structure 
of (\ref{contred}), (\ref{eqn:intrtstep_td}), (\ref{eqn:bdrytstep}) mimics the matrix structure of the second-order 
finite-volume discretization schemes. 
\end{remark}

\begin{figure}[htb]
\centering
\includegraphics[width=0.9\columnwidth]{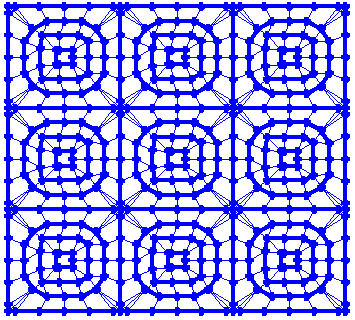}
\caption{The stencil of the S-fraction multiscale ROM of $\tA$ with $3\times 3$ coarse cells. 
Adjacent subdomains communicate through the boundary unknowns only, and interior unknowns are arranged in {\clr hidden} layers. This is schematic representation similar to the one in the Figure~\ref{fig:romsten}. See the caption thereof. }
\label{fig:romstenfll}
\end{figure}

\section{Method summary}
\label{sect:methsum}
We summarize below our method as an algorithm that is well suited for 
parallel HPC platforms.

\begin{alg}[Multiscale S-fraction ROM method]~\label{alg1red}~\\
\textbf{Stage 1 (off-line)}. In full parallel mode for each subdomain $\Omega_i$ do the following:
\begin{enumerate}
\item[(1.1)] Compute the projection subspace basis $\tns{V}_i$ and the \tcr{the projected operators 
$\tns{A}^m_i$, $\tns{B}^m_i$, $\tns{S}^m_i$ that define the ROM (\ref{tfprj})}. 
\item[(1.2)] Apply the block Lanczos algorithm with an inner product $\tns{B}^m_i$ 
to transform \tcr{$(\tns{A}^m_i, \tns{S}^m_i)$} to a block tridiagonal form
$(\tns{T}_i, \tns{R}_i)$.
\item[(1.3)] Obtain the S-fraction coefficients $\bG^k_i$, $\bhG^k_i$ 
from $(\tns{T}_i, \tns{R}_i)$ using relations (\ref{eqn:GammaTrans})
\footnote{Steps (1.2) and (1.3) can be combined using a particular form of 
block Lanczos method.}.
\item[(1.4)] Obtain $\widetilde g_{ij}$ and $\widetilde q_{ij}$ by projecting 
$g|_{\Gamma_{ij}}$ and $q|_{\Gamma_{ij}}$ onto $\tns{S}_{ij}$, respectively.
\end{enumerate}
\textbf{Stage 2 (on-line).} Starting with initial conditions 
$U^k_i|_{t=0} = \partial_t U^k_i|_{t=0} = 0$ for each time step do the following:
\begin{enumerate}
\item[(2.1)] Exchange $\widetilde{\tns{P}}_{ij}\bG^1_i (U^2_i - U^1_i)$ 
and $\widetilde{\tns{P}}_{ji}\bG^1_j (U^2_j - U^1_j)$ between the subdomains 
sharing $\Gamma_{ij}$.
\item[(2.2)] While waiting for the data exchange, compute in parallel 
for each $\Omega_i$ the updates to the interior solutions 
\tcr{$U^2_i,\ldots,U^m_i$} using (\ref{eqn:intrtstep_td}).
\item[(2.3)] Once the data exchange is complete, compute in parallel 
for each $\Omega_i$ the updates to the boundary solutions 
$U^1_i$ using (\ref{eqn:bdrytstep}). 
\tcr{Use (\ref{solatrec}) to compute the solution at receivers}.
\end{enumerate}
\end{alg}

The main computational cost of the online part in vector to matrix multiplications in $(2.2-2.3)$ is $O( \widetilde K_i^2m_i)$ 
per coarse cell per step, i.e., it grows as power $5/3$ of the total degrees of freedom (i.e., McMillan degree) in the coarse cell 
(assuming consistent discretization with $\widetilde K_i=O(m_i^2)$). Note that due to block-tridiagonal structure of 
(\ref{eqn:intrtstep_td}) step (2.2) can be efficiently palatalized further reducing the cost. Also, the order of steps 
(2.1) and (2.2) allows for what is known in computer science literature as hiding the communication latency behind the 
computations. We should point out, that the communication cost $O( \widetilde K_i)$ per step per 
cell is very low. We only exchange \emph{projected} shared boundary degrees of freedom between the adjacent subdomains 
as if we had a second order scheme, even though in our case this number is much smaller.   Such a small communication 
cost is possible because the ROMs approximate the NtD map to high (spectral) accuracy, even though 
(\ref{eqn:GammaU1st})--(\ref{eqn:GammaUrem}) resemble a three-point difference scheme.

The number of boundary basis functions $\widetilde K_i$ grows with the number of the minimal wavelengths $L$ in the 
coarse cells. The dimensionality reduction factor $\frac{N_i}{\widetilde K_i m_i}$ also grows with $L$.  So the optimal $L$ 
will be the maximum of two values: (i) the breakeven point between $O( \widetilde K_i^2m_i)$ computational and
 $O(\widetilde K_i)$ communication cost;  (ii) the optimal point of total computational cost in the entire computational 
 domain, that can be obviously estimated as $O\left( \frac{\widetilde K_i^2m_i}{L^3}\right)$ for the 3D 
 (plus one dimension of time) wave problems. The reduction is limited by the Nyquist sampling rate (two degrees of 
 freedom per wavelength), so there is no significant benefits to further increase $L$ when the ROM approaches that limit.
\begin{remark} \label{rem:solcorr}
Recall (see Figure~\ref{fig:nocornerset}) that we consider graphs with the corner nodes removed. 
We assumed that such a modification adds an insignificant error, however this error can be corrected by equating values 
$\{ \widetilde{u}_l \}_{l \in{\cal{H}}(k)}$ at the hanging nodes corresponding to the corner node $k\in\Gamma^\cap$.
The cost of the correction algorithm  (described in the Appendix~\ref{rem:solcorrap}) is insignificant compared to 
Algorithm~\ref{alg1red} in both computation and communication.
\end{remark}

\section {Stieltjesness  of the multiscale reduced order transfer function}
\label{Stieltjesness}

Let   $\tns{P}_\Gamma$ be the prolongation operator from domain $\Omega$ to the partition skeleton 
$\Gamma$ (i.e.,  $\tns{P}_\Gamma^T$ is the projection from $\Omega$ to $\Gamma$).
Assuming the ``no corners'' condition (\ref{nocorners}),  we define the skeleton transfer function as 
$$
\tns{F}(\omega^2) = \tns{P}_\Gamma^T
\left( \tns {A} + \omega^2 \tns{B} \right)^{-1}
\tns{P}_\Gamma^T \in \bfC^{K \times K},
$$ 
where $K$ is the number of nodes in the partition skeleton $\Gamma$, $K = \sum_{i,j} K_{ij}$. 
It maps from inputs to outputs satisfying condition (\ref{gsupport}), i.e.,
$$
\tns{F}(\omega^2) g|_\Gamma = u_\Gamma.
$$
Following the same reasoning as for $\tns{F}_i$ outlined in Remark~\ref{stieltjes}, 
$\tns{F}$ is the Stieltjes matrix valued function of $\omega^2$,
i.e. it can be represented as 
$$
\tns{F}(\omega^2)=\sum_{l=1}^N \frac{v_l v_l^T}{\lambda_l + \omega^2},
$$
were  $v_l = \tns{P}_\Gamma^T z_l \in \bfR^K$. It is the Stieltjes matrix-valued function of order  
$K$ according to Definition~\ref{stieltjes}.

Our multiscale reduced order model in the frequency domain can be viewed as an  approximation to 
$\tns{F}(\omega^2)$ defined as follows (assuming steps 1-3 are performed for all adjacent 
\tcr{$\Omega_i$, $\Omega_j$}).

\begin{alg}\label{alg2red}
\begin{enumerate}
\item Given input $g|_\Gamma$ compute $\widetilde{g}_{ij} = \tns{S}_{ij}^T g|_{\Gamma_{ij}}$.
\item Solve (\ref{eqn:firsttstep_conj}), (\ref{eqn:GammaUremi}) and (\ref{contred}) in terms of $\widetilde u_{ij}$.
\item Compute $u|_{\Gamma_{ij}} \approx \tns{S}_{ij} \widetilde u_{ij}$.
\end{enumerate}
\end{alg}
The following results show that the multiscale reduced order model given by Algorirthm~\ref{alg2red} is uniquely 
defined under Assumption~\ref{S-fractionrep} for Laplace frequencies satisfying conditions (\ref{noninspectr}) 
and its time-domain realization in Algorithm~\ref{alg1red} is stable and conservative.

Let $\widetilde K= \sum_{i,j} \widetilde K_{ij}$, $\widetilde{g} \in \bfR^{\widetilde{K}}$ and 
$\widetilde{u} \in \bfR^{\widetilde{K}}$ are the direct sums of respectively $\widetilde g_{ij}$ and 
$\widetilde{u}_{ij}$ for all adjacent \tcr{$\Omega_i$ and $\Omega_j$}. We define 
\tcr{$\widetilde{\tns{F}}^m(\lambda) : \bfC \to \bfC^{\widetilde{K} \times \widetilde{K}}$ via a linear map
$$
\widetilde{\tns{F}}^m(\lambda) \widetilde g = \widetilde u.
$$}

\begin{proposition}
\label{prop1stil}
If Assumption~\ref{S-fractionrep} is valid for all subdomains, then map 
$\widetilde{\tns{F}}^m(\lambda)$ is $(0,\infty)$-Stieltjes function of order $\widetilde{K}$. 
Moreover, the non-negative  definiteness can be replaced by positive-definiteness 
in the conditions 2 and 3 of Definition~\ref{stieltjes}. 
\end{proposition} 
\begin{proof}
Direct calculation yields
$$
\widetilde{\tns{F}}^m(\lambda)^{-1}=
\sum_{i=1}^{N_c} \tns{P}_{\Gamma}^{\Gamma_i} 
\left( \widetilde{\tF}_i^m(\lambda) \right)^{-1} {\tns{P}_{\Gamma}^{\Gamma_i}}^T,
$$
where the conjugate to $ \tns{P}_{\Gamma}^{\Gamma_i}\in \bfR^{\widetilde K\times  \widetilde K_i}$ is  the projection 
operator from the space of compressed inputs (outputs) supported on $\Gamma$  to the ones supported on $\Gamma_i$.
Substituting  in place of $\left( \widetilde{\tF}_i^m(\lambda) \right)^{-1}$ its representation via the S-fraction 
(\ref{S-fraction}) we obtain 
\begin{equation}
\label{tF-1}
\widetilde{\tns{F}}^m(\lambda)^{-1} =  
\sum_{i=1}^{N_c}\tns{P}_{\Gamma}^{\Gamma_i} \tns{S}_i 
\left(-{ \frac{\lambda}{\bhG_i^1} + \dfrac{1}{\frac{1}{\bG_i^1} +
\dfrac{1}{\ddots \; + \dfrac{1}{-\frac{\lambda}{\bhG_i^m} + \bG_i^m}}}} \right)
\tns{S}_i^T { \tns{P}_{\Gamma}^{\Gamma_i}}^T,
\end{equation}
where $\bG_i^k$, $\bhG_i^k$ are the coefficients of (\ref{S-fraction}) at $\Omega_i$.
Since these matrices  are symmetric positive-definite, the matrices in the brackets in (\ref{tF-1}) are symmetric 
positive-definite for $\lambda \in (-\infty,0)$ and their imaginary parts are symmetric-positive definite for 
$\Im \lambda>0$. Then the same holds for  $\left( \widetilde{\tns{F}}^m(\lambda) \right)^{-1}$ 
(thanks to non-singularity of  $\sum_{i=1}^{N_c}  \tns{P}_{\Gamma}^{\Gamma_i}{ \tns{P}_{\Gamma}^{\Gamma_i}}^T$) 
and obviously the same is true for $\widetilde{\tns{F}}^m(\lambda)$.
\end{proof}

\begin{remark}\label{MFEM}
The Stieltjes property implies stability and energy conservation of the reduced order 
model \cite{ arlinskii2011conservative}.
The strong Stieltjes property of reduced order transfer function stated in Proposition~\ref{prop1stil} is 
intrinsic for Galerkin-type approximations of symmetric nonpositive operator pencil problems. In fact, 
it can be shown, that $\widetilde F$ can be equivalently obtained via the multi-scale Galerkin spectral element 
method using the basis of block-rational Krylov subspaces (\ref{npks}) in every subdomain, {\clr e.g., see \cite{druskin2002three} for the Galerkin-equivalence result in the 1D framework. The CFL limiting  time-step  for the  Galerkin approximation  can not be smaller than the one for the exact (fine grid) problem, as confirmed numerically for our multiscale ROM  in the following section.  By exploiting this 
connection it is also possible to lift condition (\ref{gsupport}) (at the expense of some convergence speed). The rigorous analysis of the Galerkin connection is in progress}
\end{remark}
 
\section{Numerical examples}
\label{sect:numex}

We first illustrate the performance of our approach on an example of a scalar wave equation 
\be
\label{eqn:3dac}
\nabla \cdot (\sigma \nabla u) - u_{tt} = g\delta'(t)
\ee 
in a 3D cube $[0;4]^3 km^3$ with homogeneous Dirichlet boundary conditions. {\clr In all the experiments we considered $g$ supported and smoothly varied on the boundary of subdomains.} Two scenarios were considered: 
homogeneous medium with $\sigma=1 km^2/s^2$ and a heterogeneous case shown in Figure~\ref{fig:hetmed}. 
For both cases, source and receiver were centered at $(1;1.5;1.5)$ and $(3;2.5;2.5)$, respectively. 
The operator $\tA$ in (\ref{eqn:wave_fll}) was obtained a from second-order discretization of 
(\ref{eqn:3dac}) on the equidistant reference grid with $160\times 160\times 160$ nodes. {\clr The signal was convolved then with Gaussian wavelet with} cutoff 
frequency about $2\pi s^{-1}$, and, consequently, the corresponding minimum wavelength is about 
$1km$. We split the entire grid into \tcr{$4 \times 4 \times 4 = 64$} equal cubic subdomains $\Omega_i$ 
(see a 2D slice in Figure~\ref{fig:hetmed} on the right), so their edges have a length about a minimum wavelength. 
In all our experiments for all subdomains we used the same number of 
$m_i=m$ in ${\cal{V}}_{m_i}$ (equivalently, ROM layers) and the same number of boundary solutions 
$\widetilde{K}_i=\widetilde{K}$.

\begin{figure}[htb]
\centering
\includegraphics[width=0.55\columnwidth]{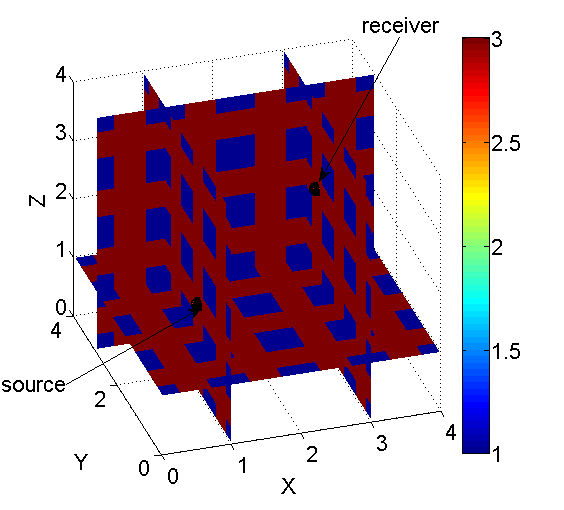}\includegraphics[width=0.45\columnwidth]{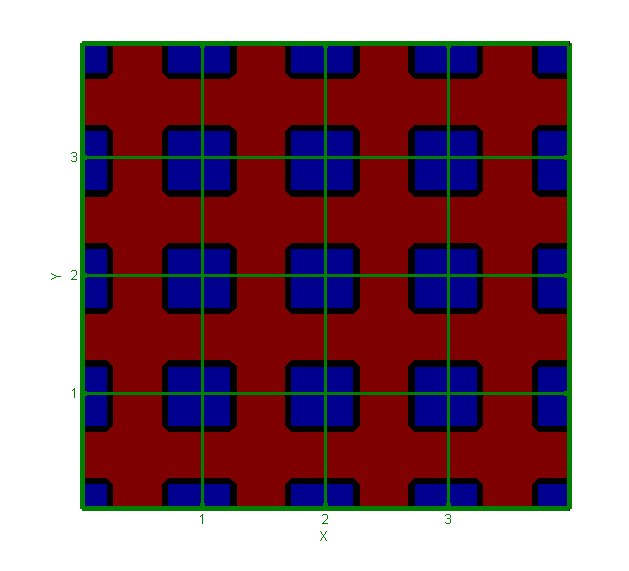}
\caption{Heterogeneous medium for the acoustic problem (\ref{eqn:3dac}) (left) and a slice of its splitting into 
subdomains shown by green lines (right).}

\label{fig:hetmed}
\end{figure}

\tcm{We start with benchmarking our MSSFROM algorithm against fine-scale solution for both scenarios. 
For multi-scale solver we used parameters $m=4$ and $\widetilde{K}=180$. The results of simulations are shown 
in Figure~\ref{fig:scal_sol}. For the homogeneous case, serial computation times were $0.2$ hours and $2.2$ hours 
for the multi-scale and full-scale solvers, respectively, and CFL numbers were approximately the same. 
For the heterogeneous scenario, the serial computation time for the multi-scale solver remained the same $0.2$ hours, 
while the full-scale solver required $3.3$ hours. That is caused by the difference in CFL numbers: for the full-scale 
solver it was lower by a factor of $1.5$.}
\begin{figure}[htb]
\centering
\includegraphics[width=0.45\columnwidth]{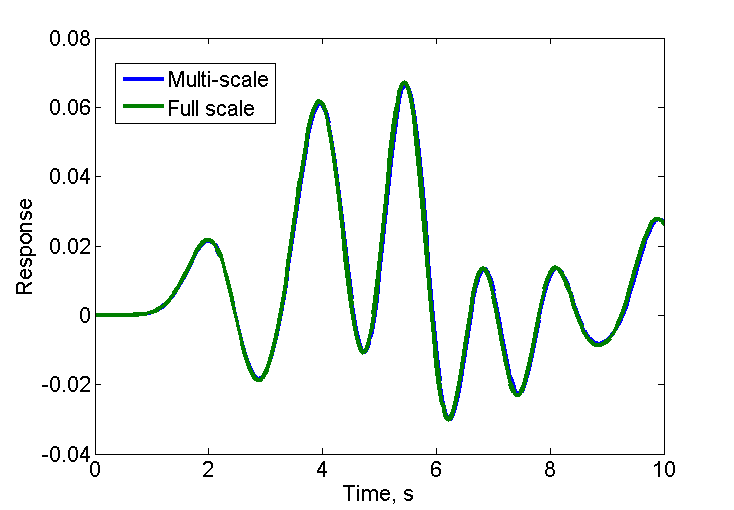}
\includegraphics[width=0.45\columnwidth]{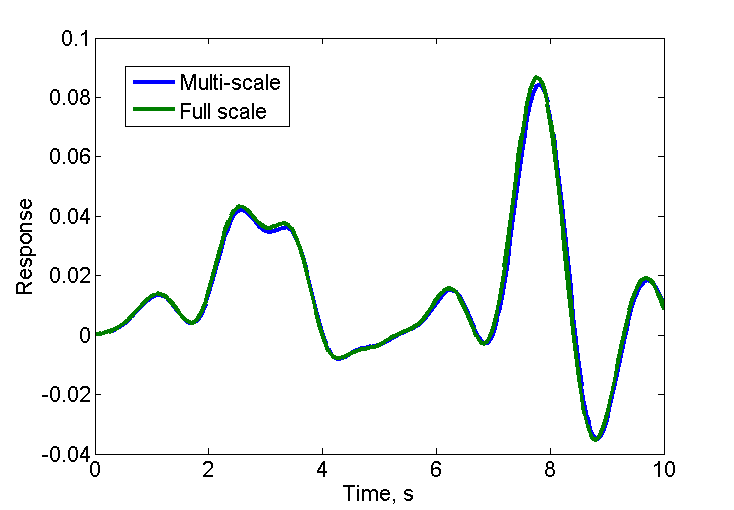}
\caption{Simulations results using MSSFROM and full-scale solver for homegeneous (left) and heterogenous (right) media}
\label{fig:scal_sol}
\end{figure}

{\clr In our next example we perform the numerical dispersion study of MSSFROM. For homogenous case scenario with unit speed $\sigma=1$, we fixed parameters $m=4$ and $K=150$ and changed the cut-off frequency $\omega_{max}$.  We note that, while the size of subspace remained the same, the shape of the boundary basis functions (and, consequently, the subspace ${\cal{V}}_m$) depended on the frequency range under consideration (see equation \ref{trgrstep}). The number of degrees of freedom used for approximating the solution at each coarse cell is $mK$. Hence, assuming the coarse cell edge is of unit length, $\frac{2\pi\sigma\sqrt[3]{mK}}{\omega_{max}}$ degrees of freedom (ppw) were used per wavelength. In Figure~\ref{fig:disp} we plot the error of MSSFROM with respect to the true fine-scale solution without corner set removal. As one can observe, MSSFROM provides good accuracy up to about 4 ppw.
\begin{figure}[htb]
\centering
\includegraphics[width=0.7\columnwidth]{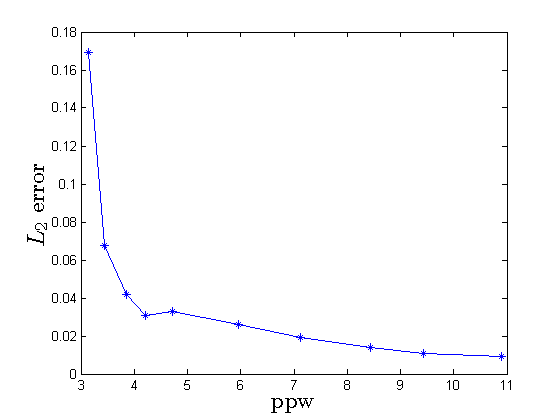}
\caption{$L_2$ error with respect to the number of degrees of freedom per wavelength. For coarse cell edge of unit length that number can be estimated as $\frac{2\pi\sigma\sqrt[3]{mK}}{\omega_{max}}$.}
\label{fig:disp}
\end{figure}
}

Next, we perform the convergence study of the MSSFROM approach. In the first set of experiments we kept the subspace of boundary solutions fixed with $\widetilde{K}=294$ and 
increased $m$. In Figure~\ref{fig:ord_incr_maxerr} we show the maximum absolute error for different value of $m=2,3,4,5$. 
For both scenarios the algorithm demonstrates exponential convergence. Also, for the heterogeneous case $\sigma$ 
varies from $1$ to $3$ while for the homogenous one it is equal to $1$. Therefore, a smaller subspace is required to 
approximate the former case. This explains why the convergence is faster in the heterogeneous case compared to the
homogeneous scenario. 
\begin{figure}[htb]
\centering
\includegraphics[width=0.7\columnwidth]{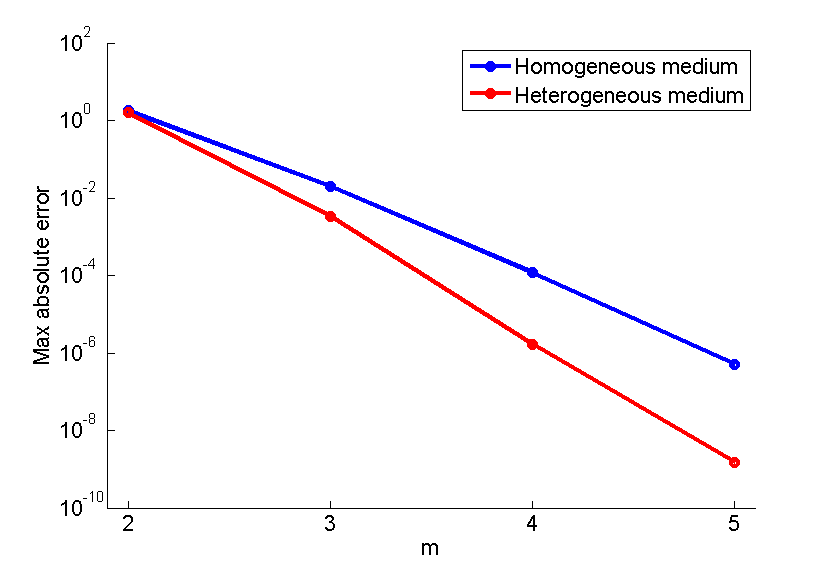}
\caption{Convergence curves with respect to $m$ for homogeneous and heterogeneous cases. 
While both show exponential convergence, the latter is faster because of inhomogeneity.}
\label{fig:ord_incr_maxerr}
\end{figure}

In the next experiment we fixed $m=4$ and were increasing $K_i=K$ for all $i$ simultaneously. Since the length of each 
edge of $\Omega_i$ is about a minimum wavelength, $\sqrt{K/6}$ shows how many degrees of {\clr freedom per wavelength are used for approximating boundary solutions.} In Figure~\ref{fig:bndfnc_incr_maxerr_het_hom_80ref} we show the errors for the homogeneous and 
heterogeneous cases versus $\sqrt{K/6}$. In both cases convergence is exponential, with plateaus while $\sqrt{K/6}$ changes from one integer to the consequent one. 
We also note that the error for the heterogeneous case is worse than for the homogeneous one because of the 
formulation we used to compute boundary solutions. Indeed, for a given face of the subdomain we assumed that 
the medium is uniform in the direction normal to the face. For the heterogeneous case it slightly worsens the approximation 
properties of the boundary solutions subspace. 
\begin{figure}[htb]
\centering
\includegraphics[width=0.7\columnwidth]{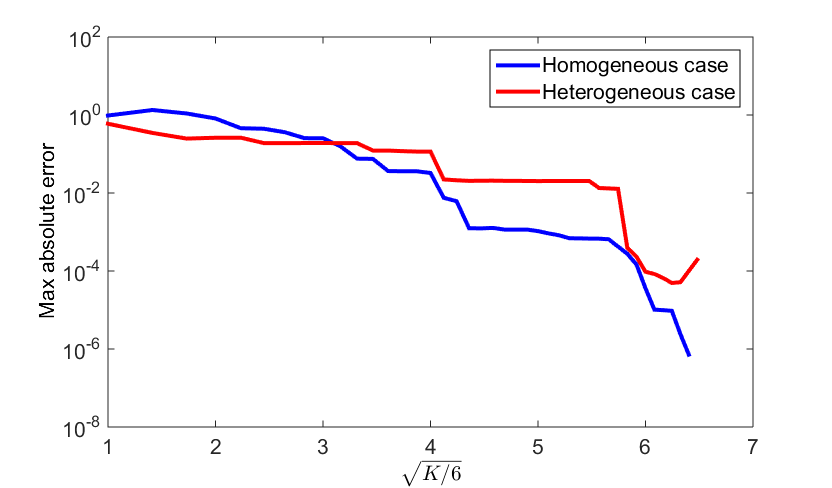}
\caption{Convergence curves for homogeneous and heterogeneous cases. 
Both scenarios show exponential convergence, with plateaus while $\sqrt{K/6}$ changes from one integer to the 
next one. Here $\sqrt{K/6}$  is approximately equal to the number of degrees of the freedom per wavelength for the coarse cell boundary discretization. Heterogeneous case converges slower because of the approximation we used to construct the boundary solutions.}
\label{fig:bndfnc_incr_maxerr_het_hom_80ref}
\end{figure}

In our next experiment we consider a 3D anisotropic elastodynamic problem {\clr
\bea
u_{i,tt}=\sigma_{ij,j},\\
\sigma_{ij}=C_{ijkl}(u_{k,l}+u_{l,k})/2
\eea
in the cube $[0,5]^3 km^3$ with homogeneous 
Dirichlet boundary conditions. Here $u_i$ and $\sigma_{ij}$ are components of displacement vector and stress tensor, respectively. $C_{ijkl}$ are elastic moduli tensor components. In isotropic medium it has a form $C_{ijkl}=\lambda\delta_{ij}\delta_{kl}+\mu(\delta_{ik}\delta_{jl}+\delta_{il}\delta_{jk})$.} The locations of $x$-oriented stress source receiver measuring $x-$ component of the displacement were $(1;2.5;2.5)$ and $(4;2.5;2.5)$, respectively. {\clr The medium (see Figure~\ref{fig:elasfig}) consists of water-filled fracture and cavities (blue) with $\frac{\lambda}{\rho}=2\frac{km^2}{s^2}$, $\frac{\mu}{\rho}=0$ and empty air-filled cavity (white) with $\frac{\lambda}{\rho}=0$, $\frac{\mu}{\rho}=0$ embedded in slow isotropic background (cyan) with $\frac{\lambda}{\rho}=2\frac{km^2}{s^2}$ and $\frac{\mu}{\rho}=1\frac{km^2}{s^2}$. The background is surrounded by fast anisotropic orthorhombic (red) layers with non-zero components of elastic moduli tensor $\frac{C_{1111}}{\rho}=8\frac{km^2}{s^2}$, $\frac{C_{1122}}{\rho}=6\frac{km^2}{s^2}$, $\frac{C_{1133}}{\rho}=5\frac{km^2}{s^2}$, $\frac{C_{2222}}{\rho}=9\frac{km^2}{s^2}$, $\frac{C_{2233}}{\rho}=4\frac{km^2}{s^2}$, $\frac{C_{3333}}{\rho}=10\frac{km^2}{s^2}$, $\frac{C_{1212}}{\rho}=1\frac{km^2}{s^2}$, $\frac{C_{2323}}{\rho}=2\frac{km^2}{s^2}$, $\frac{C_{1313}}{\rho}=3\frac{km^2}{s^2}$. The minimum wavelength at cutoff frequency is approximately $1km$.} The fine reference grid 
with $200\times 200 \times 200$ nodes has been split into $5\times 5\times 5 = 125$ subdomains. To construct the ROMs, 
we used $m=4$ and $K=75$ for all subdomains, i.e. we ended up with $6\times 75\times 4=1800$ degrees of freedom 
per subdomain instead of $40\times 40\times 40\times 3=192000$ for the full-scale reference discretization. 
\begin{figure}[htb]
\centering
\includegraphics[width=0.55\columnwidth]{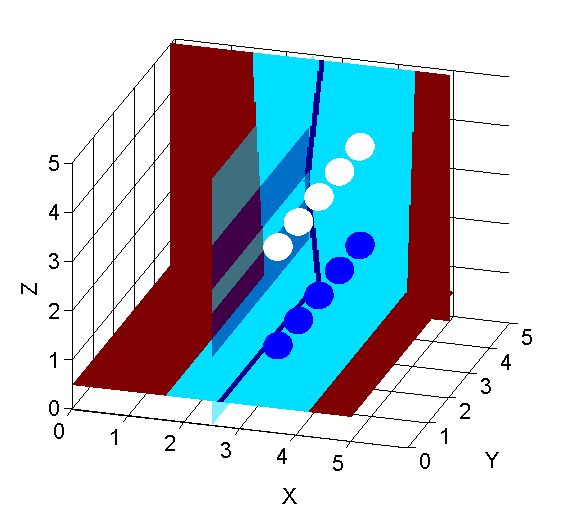}\includegraphics[width=0.45\columnwidth]{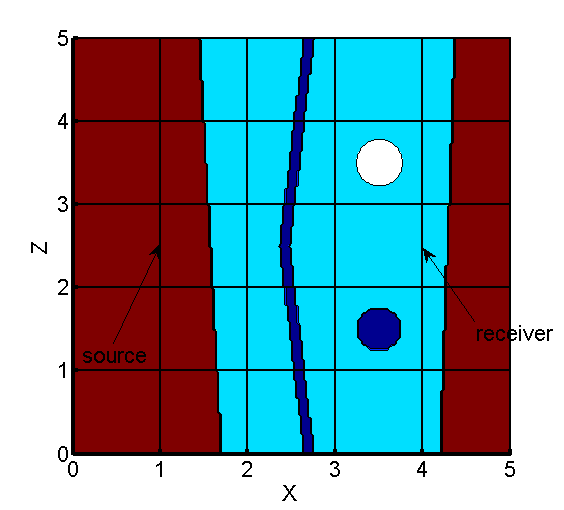}
\caption{3D medium (left) and its slice at $y=2.5 km$ for elastodynamic benchmark. The formation consists of slow 
isotropic (cyan) and fast anisotropic (red) parts with water-filled fracture and cavity (blue) and air-filled cavity (white).
}
\label{fig:elasfig}
\end{figure}

\tcr{A comparison of the solution at the receiver (see Figure~\ref{fig:elasfig} for source and receiver positions) between
our multiscale method and the reference solution is given in Figure~\ref{fig:elasres}. Again, here we report the CPU
times for serial implementations of both methods.}
Though the reduction in the number of unknowns was more than two orders of magnitude, the speed up in the simulation 
time (on-line part only vs the reference full-scale simulation) was less: 0.26 hour vs 10.5 hours. 
That is caused by the dense structure of matrices $\bhG^k$ and $\bG^k$ which results in larger computational 
cost (per unknown) of the matrix multiplication in the time-stepping scheme 
(\ref{eqn:intrtstep_td}), (\ref{eqn:bdrytstep}).  \tcm{We also note that CFL numbers for multi-scale and full-scale solvers were approximately the same in this example.}
\tcr{Even though the speedup in the serial case is already substantial, the parallel implementation will favor our
multiscale approach even more due to low ratio of communication and computation costs, as it 
was emphasized in Section~\ref{sect:methsum}. An efficient parallel implementation and testing is a topic of future work}.

\begin{figure}[htb]
\centering
\includegraphics[width=0.7\columnwidth]{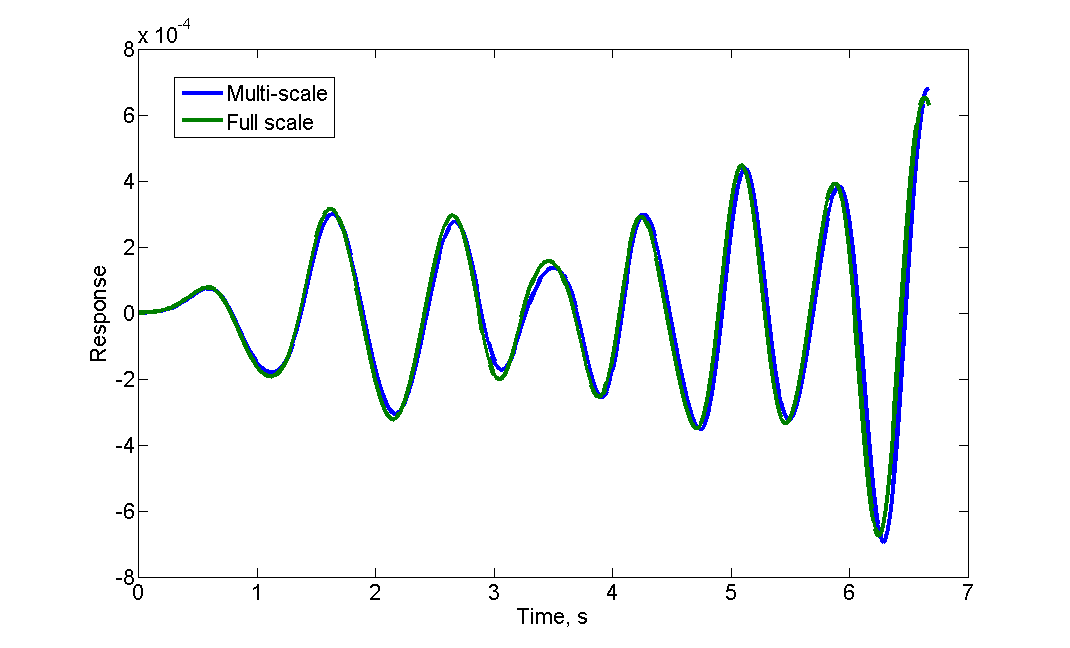}
\caption{3D anisotropic elastodynamic problem. Multi-scale solver against full-scale solver.}
\label{fig:elasres}
\end{figure}

\section{Conclusion}
We developed a multi-scale reduced-order modeling approach for efficient large wavefield simulations. Similar to the
conventional multi-scale methods, our algorithm consists of the off-line and the on-line stages. At the first stage a 
reduced-order multi-scale model is constructed. In particular, we start with the fine grid discretization of the spatial 
operator and split the grid into subdomains (coarse cells). The divide-and-conquer-type algorithm was suggested to 
construct the corresponding partitioning of the operator. Then we reduce the number of degrees of freedom in the 
partitioning by substituting it with a high-accuracy reduced-order model. The crucial step of our approach is the 
transformation of the reduced-order model to S-fraction form. This results in the sparsified (block-tridiagonal) stencil 
of the obtained discrete spatial operator, reduction of communications between the subdomains and 
\tcr{greatly reduced overall computational cost}. The off-line stage is embarrassingly parallel and is performed just 
once for the entire time-domain simulations and for all sources. At the on-line stage, the time-domain simulations 
are performed within the obtained multi-scale ROM formulation. The significantly reduced number of unknowns and 
the sparsified stencil results in a speed up of our approach by a factor of up to 30 compared to full-scale simulations 
even on serial computer. Due to the reduced communications between subdomains, we expect even more dramatic 
results when an MPI and/or GPU implementation is ready.

Besides the high performance computing implementation, there remains a number of open algorithmic and theoretical 
problems, successful resolution of which may lead to a better understanding of the developed approach and its further 
improvements. Some of these issues are listed below. 
\begin{itemize}
\item Removal of the restriction on inputs and outputs to be supported only at the boundary of course cells.
That can be done by establishing a connection between our approach and conventional multiscale and spectral 
element methods.
\item Making the Courant-Friedrichs-Lewy (CFL) stability condition consistent with the Nyquist limit corresponding to 
the source frequency bandwidth.
\item Obtaining compact representation of matrix S-fraction coefficients, and as result, further decrease the 
computational cost of the time-stepping (online stage).
\end{itemize}
\section{Acknowledgments}
We are grateful to Smaine Zeroug for bringing our attention to the large scale modeling in seismic exploration, 
Leonid Knizhnerman and Yousef Saad for their invaluable contribution to the POD algorithm for the construction of 
boundary subspaces \cite{DKSZ} and also Alexander Aptekarev and Yuri Dyukarev for introducing work 
\cite{dyukarev2004indeterminacy}.

\appendix

\section{Divide and conquer domain partitioning}
\label{sect:divconq} 
\subsection{Two subdomains}
\label{sect:divconq_2sd}

The adjacent nodes  satisfy the following symmetry property that follows straightforwardly from symmetry of $\tA$.
\begin{lemma}\label{lem1}
If $\cA(\mO')\cap \mO''=\emptyset$ then $\cA(\mO'')\cap \mO'=\emptyset$.
Also, since the diagonal entries of $\tA$ are assumed to be all non-zero, we have $\mO'\subset\cA(\mO')$ for any $\mO'$.
\end{lemma}
\begin{prop}\label{prop1}
Let there exist a nonempty $\mO  \subsetneq  \Omega$, such that $\cA(\mO)\subsetneq  \Omega$. 
Then there exist 
\tcr{disjoint} subdomains $\mO_1,\mO_2\ne \emptyset$ and $\Gamma$ satisfying conditions 
$$
\Omega = \mO_1\cup \mO_2\cup\Gamma, \mbox{ and }  \Gamma=\cA(\mO_1)\setminus \mO_1=\cA(\mO_2)\setminus \mO_2. 
$$
\end{prop}
\begin{proof}
The proof is constructive. We choose $\mO_2=\Omega\setminus \cA(\mO)$, 
$\mO_1=\Omega\setminus \cA(\mO_2)$, $\Gamma=\Omega\setminus \mO_1\setminus \mO_2$. 
By construction $\mO_1,\mO_2\ne \emptyset$, all three subdomains \tcr{are disjoint} and $\Omega$ is their union.
Then 
$$
\Gamma=\Omega\setminus \mO_1\setminus \mO_2=
\Omega\setminus (\Omega\setminus \cA(\mO_2))\setminus \mO_2=
\cA(\mO_2)\setminus \mO_2.
$$
By construction, $\mO_1$ is a complement to $\cA(\mO_2)$, so from Lemma~\ref{lem1} we conclude, 
that $\cA(\mO_1)\cap \mO_2=\emptyset$. Hence, $\cA(\mO_1)\setminus \mO_1\subset\Gamma$ and, to show that 
$\Gamma=\cA(\mO_1)\setminus \mO_1$, it remains to prove that all elements of $\Gamma$ are in $\cA(\mO_1)$. 
Again, from the definition of $\mO_2$ and Lemma~\ref{lem1} it follows that $\mO\cap\cA(\mO_2)=\emptyset$ and, 
consequently, $\Gamma\cap \mO=\emptyset$. Therefore, since $\Gamma=\cA(\mO)\setminus\mO_1$, we have 
$\mO_1\supseteq \mO$ and, obviously, $\cA(\mO_1)\supseteq \cA(\mO)$. 
By construction $\Gamma\subset \cA(\mO)$, so $\Gamma\subset \cA(\mO_1)$.
\end{proof}

Obviously, any $\tA$ with at least one trivial non-diagonal element satisfies the condition of Proposition~\ref{prop1}.
The  case  $\Gamma=\emptyset$ yields decoupled problems on $\mO_1$ and $\mO_2$.
The proof of the Proposition~\ref{prop1} provides an algorithm for partitioning of a {\it sparse} matrix $\tA$.
We denote by $\Omega_{i} = \mO_i \cup \Gamma$ of size (number of nodes) $N_i$ and by 
$\tA^i \in \bfR^{N_i \times N_i}$ (with elements $a^i_{kl}$) the split operators defined on the nodes of 
$\Omega_i$, $i=1,2$ defined via the following three step algorithm. 

\begin{alg}
\label{alg1}
According to Proposition~\ref{prop1}, $a_{kl}=0$ if $k$ and $l$ belong to different $\mO_i$, 
so such elements are excluded.
\begin{enumerate}
\item  For all  pairs of indices $(k,l)$,  such that at least one of them belongs to
$\mO_i$, $i = 1,2$ set 
$$
a_{kl}^i = a_{kl}.
$$
\item For all  $k \ne l$, such that $k,l \in \Gamma$, 
set 
$$
a_{kl}^1 = \alpha_{kl} a_{kl}, \quad a_{kl}^2 = (1-\alpha_{kl}) a_{kl},
$$ 
where $0 \le \alpha_{kl} = \alpha_{lk} \le 1$.
\item For $k \in \Gamma$ set $a_{kk}^i = a_{kk} \dfrac{s_i}{s_1+s_2}$, $i = 1, 2$,
where 
$$
s_i = \sum_{l \ne k} |a_{kl}^i|.
$$
\end{enumerate}
\end{alg}
For the splitting of (\ref{eqn:wave_fll_fd}) we also need the splitting of \tcr{$\tB^i = \diag \{b_k^i\}$}
via a simplified version of Algorithm~\ref{alg1} performed as follows
\tcr{\be
\label{alg1B}
b_k^i = \left\{ \begin{tabular}{ll}
$b_k$, & if $k \in \mO_i$, \\
$((i-1) + (-1)^{i+1} \beta_k) b_k$, & otherwise.
\end{tabular} \right.
\ee}
where $0 < \beta_k < 1$ and $i = 1, 2$.
Normally, $\alpha_{ij}$ and $\beta_i$ are chosen to approximately balance the condition 
numbers of partitioned problems.

By construction, matrices $\tns{A}^i$, $\tB^i$ obtained by Algorithm~\ref{alg1} satisfy the identity 
\tcr{$$
\tA + \omega^2 \tB = \sum_{i=1}^{2}{\tns{P}_i (\tA^i + \omega^2 \tB^i) \tns{P}_i^T}, 
\qquad \forall \omega^2\in \bfC. 
$$}

For the stability of the reduced order model it is important that $\tA^i$ inherit 
the non-positive definiteness of $\tA$. The following proposition shows an even stronger 
result for the important class of diagonally dominant matrices arising from discretization of second order elliptic operators. 
\begin{prop}
\label{prop2}
Let $\tA$ be diagonally dominant, i.e., $\forall k$  $|a_{kk}| \ge \sum_{l \ne k} |a_{kl}|$.
Then $\tA^i$, $i = 1, 2$ obtained via Algorithm~\ref{alg1} are also diagonally dominant non-positive 
 definite  matrices. 
\end{prop}
\begin{proof}
According to Proposition~\ref{prop1}, if $l \in \mO_i$ then $\forall k$ $a_{kl}^i = a_{kl}$ .
For the rows with diagonals $k$ in $\Gamma$, from the dominance of the corresponding row of $\tA$ and 
from Algorithm~\ref{alg1} we obtain diagonal dominance
$$
\sum_{l \ne k} |a^i_{kl}| = \frac{s_i}{s_1+s_2} \sum_{l \ne k} |a_{kl}| \le 
\frac{s_i}{s_1+s_2} |a_{kk}| = |a_{kk}^i|.
$$
From the non-positive  definiteness and diagonal dominance of $\tA$ its diagonal is negative and 
so are the diagonals of $\tA^i$, that together with diagonal dominance of the latter yields their non-positive
definiteness.
\end{proof}

\subsection{Multidomain algorithm}\label{sect:divconq_mltsd}
\tcr{
We employ Algorithm~\ref{alg1} as an elementary step of our ``divide and conquer'' graph partitioning algorithm
presented below.}
\begin{alg}\label{alg2}

\tcr{Initialize the partition with $\Omega_1=\Omega$, $\tA_1=\tA$, $\tB_1=\tB$.}

\tcr{For $p=1,2,\ldots$ perform the following:}
\begin{enumerate}
\item  From $\Omega_i$, $i=1,\ldots,p$ choose $\Omega_{i'}$ to be partitioned 
such that $\Omega_{i'}$ satisfies the conditions of Proposition~\ref{prop1}, i.e., it is splittable by Algorithm~\ref{alg1}.
\item Renumber $\Omega_i$ and $\tA_i$, so that $\Omega_{i'}$, $\tA_{i'}$ and $\tB_{i'}$ become  
$\Omega_{p}$, $\tA_{p}$ and $\tB_{p}$ respectively.
\item Set (with some abuse of notation) $\Omega=\Omega_{p}$.  Find $\mO_1$, $\mO_2$ and $\Gamma$ 
satisfying the conditions of Proposition~\ref{prop1}. 
\item Set $\tA=\tA_p$ and apply Algorithm~\ref{alg1} and (\ref{alg1B}) to obtain 
\tcr{$\tA^1$, $\tA^2$ and $\tB^1$, $\tB^2$} respectively. 
\item Set $\Omega_{p} = \Omega_1$, $\Omega_{p+1}=\Omega_2$, 
\tcr{$\tA_{p} = \tA^1$, $\tA_{p+1} = \tA^2$, $\tB_{p} = \tB^1$, $\tB_{p+1} = \tB^2$}.
\end{enumerate}
\end{alg} 

The partitioning is repeated until the subdomain size is reduced to a desirable level,
or until $\tA_i$ become full. We assume that after $N_c-1$ steps of Algorithm~\ref{alg2} we partition $\tA$
into submatrices $\tA_i\in\bfR^{N_i\times N_i}$, $i=1,\ldots,N_c$ satisfying (\ref{splitN_c}).
Note that Proposition~\ref{prop2} can be recursively extended to $\tA_i$, $i=1,\ldots,N_c$.


{\clr \section{Corner set removal algorithm}\label{sect:cor_set_rem}
          
Let node $k$ belong to the ``corner'' set. Then for each $\Gamma_{ij}$ such that 
$k \in \Gamma_{ij}$ we introduce a hanging node $k_{ij}$ by copying $k$ and all elements $\left(\tns{P}_i\tns{A}_i\tns{P}_i^T+\tns{P}_j\tns{A}_j\tns{P}_j^T\right)_{kl}$ and $\left(\tns{P}_i\tns{B}_i\tns{P}_i^T+\tns{P}_j\tns{B}_j\tns{P}_j^T\right)_{kl}$. In particular, we define the elements of the modified matrix $\tA'$ as well as diagonal matrix $\tB'$  as
$$(\tA')_{k_{ij}l}=(\tA')_{lk_{ij}}=\left(\tns{P}_i\tns{A}_i\tns{P}_i^T+\tns{P}_j\tns{A}_j\tns{P}_j^T\right)_{kl}$$
and 
$$(\tB')_{k_{ij}k_{ij}}=\left(\tns{P}_i\tns{B}_i\tns{P}_i^T+\tns{P}_j\tns{B}_j\tns{P}_j^T\right)_{kk}.$$
Let ${\cal{H}}(k)$ be a set of all the constructed nodes with parent node $k$. 
Once the process is finished for all $\Gamma_{ij}$ containing node $k$, we remove node $k$ and all corresponding elements 
of $\tA$ and $\tB$.

}
\section{Computation of matrix S-fraction coefficients via block-Lanczos algorithm}
\label{appC}

Matrix S-fraction coefficients $\bG^k$ and $\bhG^k$ can be computed by applying a block
version of Lanczos iteration. We use Lanczos approach because it allows to implement the well 
developed techniques of a standard block-Lanczos algorithm \cite{Montgomery1995} to the 
matrix pencil $(\tns{A}^m, \tns{B}^m)$ to obtain a unitary 
$\tns{Q} \in \bfR^{\widetilde{K} m \times \widetilde{K} m}$ such that 
\be
\widetilde{\tns{F}}^m(\omega^2) = \tns{R}^T 
\left( \tns{T} + \omega^2 \tns{I} \right)^{-1} \tns{R},
\label{eqn:NtDtri}
\ee
with 
\begin{equation}
\tns{T} = \tns{Q}^T {\tA}^m\tns{Q}, \quad
\tns{R} = \tns{Q}^T \tns{S}^m = [\tns{\beta}_1, 0, 0, \ldots, 0 ]^T.
\end{equation}
Here $\tns{T}$ is a Hermitian block tridiagonal matix
\tcr{
$$
\tns{T} = 
\begin{bmatrix}
\tns{\alpha}_{1}& \tns{\beta}_{2} &0 & 0& \ldots & 0\\
\tns{\beta}_{2}& \tns{\alpha}_{2} &\tns{\beta}_3 & 0&  \ldots & 0\\
\vdots &\vdots& \vdots& \vdots &\ddots & \vdots\\
0& \ldots & 0 & \tns{\beta}_{m-1} & \tns{\alpha}_{m-1} & \tns{\beta}_{m} \\
0& \ldots & 0 & 0 & \tns{\beta}_{m} & \tns{\alpha}_{m} \\
\end{bmatrix}
$$}
with blocks $\tns{\alpha}_k = \tns{\alpha}^T_k \in \bfR^{\widetilde{K} \times \widetilde{K}}$ 
and $\tns{\beta}_k = \tns{\beta}^T_k \in \bfR^{\widetilde{K} \times \widetilde{K}}$, 
\tcr{$k = 1,2,\ldots,m$}.

The transfer function (\ref{eqn:NtDtri}) can be expressed as
\be
\widetilde{\tns{F}}^m  = \tns{\beta}_1 \tns{W}_1,
\label{eqn:NtDW}
\ee 
where matrices $\tns{W}_k \in \mathbb{R}^{\widetilde{K} \times \widetilde{K}}$, 
\tcr{$k = 1, 2, \ldots, m+1$}
satisfy the three-term relations that make apparent the connection to finite-difference schemes:
\begin{equation}
\begin{split}
\tns{\alpha}_1 \tns{W}_1 + 
\tns{\beta}_2 \tns{W}_2 + \omega^2 \tns{W}_1 & = \tns{\beta}_1, \\
\tns{\beta}_k \tns{W}_{k-1} + 
\tns{\alpha}_k \tns{W}_k + 
\tns{\beta}_{k+1} \tns{W}_{k+1} + \omega^2 \tns{W}_k & = \tns{0},
\end{split}
\label{eqn:tridiagW}
\end{equation}
with $\tns{W}_{m+1} = 0$.

A second change of coordinates can simplify (\ref{eqn:NtDW}) even further. Specifically, we can  
transform (\ref{eqn:tridiagW}) to (\ref{eqn:GammaUrem}). The corresponding transformation is 
block-diagonal and is performed recursively for $k = 1, 2, \ldots, m$ using
\tcr{
\begin{equation}
\begin{split}
\bhG^{k} & = \tns{G}_{k} {\tns{G}_{k}}^T\\
\bG^{k} & = - {\tns{G}_{k}}^{T}\tns{\alpha}_{k} {\tns{G}_{k}}^{-1} - \bG^{k-1}, \\
\tns{U}^{k} & = \tns{G}_{k} \tns{W}_{k},\\
\tns{G}_{k+1} & = \left[ {\tns{G}_k}^{T} \bG^k \right]^{-1} \tns{\beta}_{k+1}, 
\end{split}
\label{eqn:GammaTrans}
\end{equation}}
starting with $\tns{G}_1 = \tns{\beta}_1$, $\bG^0 = 0$. 
\tcr{After this transformation the transfer function becomes}
$$
\widetilde{\tns{F}}^m(\omega^2)  = \tns{U}^1 = 
\tns{E}_1^T \left( \tns{L}+\omega^2 \tns{I} \right) \tns{E}_1 \bhG^1
$$
where $\tns{E}_1= [\tns{I}, 0, 0, \ldots, 0 ]^T\in\bfR^{\widetilde{K}m\times\widetilde{K}}$ 
\tcr{and $\tns{L}$ is the block-tridiagonal difference operator defined by 
(\ref{eqn:GammaU1st})--(\ref{eqn:GammaUrem}). }

\section{Correction for corner set nodes}
\label{rem:solcorrap}

For any $i \ne j$ the node-wise MSSFROM solution $\widetilde{u}$ at the boundary nodes of $\Gamma_{ij}$ can 
be obtained as $\tns{S}_{ij} \widetilde{u}_{ij}$. Indeed, by construction of the modified graph without the corner set, 
we note that $\sum_{l \in {\cal{H}}(k)}{\widetilde{u}_{l,tt} b_{ll}}$ must approximate the fine-scale balance 
equation at the corner set node $k$, i.e. the term $u_{k,tt} b_{kk}$. Hence, assuming $\widetilde{u}_l$ is accurate 
at previous and current time step, we end up with the following correction procedure:

\tcm{
\begin{enumerate}
\item For all $k \in \Gamma^\cap$ compute MSSFROM solutions $\widehat{u} \in \bfR^{K_h}$ at hanging nodes 
$l \in {\cal{H}}(k)$ (we denote by $K_h$ the total number of hanging nodes): 
$$
\widehat{u}_ l =e_l^T \tns{S}_{ij} \widetilde{u}_{ij},
$$ 
where $l \in \Gamma_{ij}$ and $e_l \in \bfR^{K_{ij}}$ is the unit vector with component $1$ at node $l$.
\item For all $k \in \Gamma^\cap$ define MSSFROM solution at the corner set node $k$ as 
$\frac{1}{b_{kk}} \sum_{l' \in{\cal{H}}(k)}{\widehat{u}_{l'} b_{l'l'}}$ 
and then compute the correction term $\delta \widehat{u} \in \bfR^{K_h}$ at hanging nodes $l \in {\cal{H}}(k)$ as 
$$
\delta \widehat{u}_l = \frac{1}{b_{kk}}\sum_{l' \in {\cal{H}}(k)}{\hat{u}_{l'} b_{l'l'}} - \widehat{u}_l.
$$
\item Update $\widetilde{u}_{ij}$ by adding the correction term 
$$
\delta \widetilde{u}_{ij} = \tns{S}^T_{ij}{\tns{P}^H_{ij}}^T \delta \widehat{u},
$$ 
where $\tns{P}^H_{ij} \in \bfR^{K_h \times K_{ij}}$ is the prolongation operator from $\Gamma_{ij}$ to 
hanging nodes $\cup_{k \in \Gamma^\cap}{\cal{H}}(k)$: 
\be
\left( \tns{P}^H_{ij} x \right)_k = 
\left\{ \begin{tabular}{ll}
$x_k$, & if  $k \in \Gamma_{ij}$ \\
$0$, &  otherwise
\end{tabular}\right. , 
\quad \mbox{for } x \in \bfR^{K_{ij}}.
\ee
\end{enumerate}}
\bibliography{siam_sfmsfv}
\end{document}